\documentclass[12pt]{amsart}
\usepackage{amsmath, amssymb, amsthm,mathrsfs}
\usepackage{fullpage}
\usepackage{geometry} 
\usepackage{hyperref} 
\usepackage{stmaryrd}
\usepackage{graphicx} 
\usepackage{xcolor}
\usepackage[
backend=biber,
style=alphabetic,
]{biblatex}
\usepackage{tikz-cd}
\usepackage{biblatex} 
\theoremstyle{plain}
\newtheorem{theorem}{Theorem}[section]
\newtheorem{lemma}[theorem]{Lemma}
\newtheorem{corollary}[theorem]{Corollary}

\newtheorem{conjecture}[theorem]{Conjecture}

\theoremstyle{definition}
\newtheorem{definition}[theorem]{Definition}
\newtheorem{example}[theorem]{Example}
\newtheorem{Lthm}{Theorem}

\newtheorem{Lcor}[Lthm]{Corollary}

\theoremstyle{remark}
\newtheorem{remark}[theorem]{Remark}

\newcommand{\R}{\mathbb{R}}

\title{TRIGONAL CURVE WITH TRIGONAL DEFORMATION OF MAXIMAL RANK}
\author{Jiacheng Zhang}
\date{\today}
\usepackage{microtype}
\begin{document}

\begin{abstract}
By extending methods of Favale-Pirola \cite{FavalePirola2022} and Gonz\'{a}lez-Alonso-Torelli\cite{TorelliGonzalezAlonso2024} to toric surfaces via the toric Jacobian ring, we give a computational criterion to determine whether a nondegenerate curve given by ample section of a Hirzebruch surface has first order deformation of rank $g$ within the linear system. We are able to show there exist trigonal curves with trigonal deformations of rank $g$ for $g=2k+1$ and $k\ge 2$ by computing an explicit example.
\end{abstract}


\maketitle

\section{Introduction}
When one studies how a smooth curve $C/\mathbb{C}$ deforms within a smooth family, i.e., $\pi:\mathcal{C}\rightarrow (B,0)$ with $\pi^{-1}(0)=C$ and $B$ a complex-analytic variety with all fibers being smooth curves, it is natural to consider how the Hodge structure varies within the same family. In particular, one may consider the period map $P:B\rightarrow \Gamma\backslash D$ that sends each $b\in B$ to the polarized Hodge structure of the curve $\pi^{-1}(b)=\mathcal{C}_b$, where $D$ is the classifying space of all polarized Hodge structures on $H^1(C,\mathbb{Z})$ and $\Gamma$ is a subgroup of $\operatorname{Aut}(H^1(C,\mathbb{Z}),\langle\cdot \rangle)$ with $\langle\cdot \rangle$ being the polarization. This is called the variation of Hodge structure (See Chapter 10 of \cite{Voisin2002} and also \cite{Griffiths68}). One of the first steps one can take to study this map is to linearize it by studying the induced map of tangent spaces\cite{CarlsonGreenGriffithsHarris1983}. That is the map 
\begin{equation}\label{infperiod}
    dP_s:T_{B,0}\rightarrow \operatorname{Hom}(H^0(C,\omega_C),H^1(C,\mathcal{O}_C))
\end{equation}
This infinitesimal period map roughly speaking is the Infinitesimal Variation of Hodge structure. In \cite{Griffiths68} (See also p.253 of \cite{Voisin2002}), Griffiths showed that this map is the composition of the Kodaira-Spencer map $\rho:T_{B,0}\rightarrow H^1(C,T_C)$, which parametrizes first order deformations of $C$, with the map given by cup-product followed by contraction
\[H^1(C,T_C)\rightarrow \operatorname{Hom}(H^0(C,\omega_C),H^1(C,\mathcal{O}_C))\]
Moreover, for $\xi\in T_{B,0}$ and $\varphi,\psi\in H^0(C,\omega_C)$, we have 
\begin{equation}\label{symmetry}
    \langle \rho(\xi)\cdot \varphi,\psi\rangle+\langle \varphi,\rho(\xi)\cdot \psi\rangle=0
\end{equation}
This enables one to consider each element of $T_{B,0}$ as a linear map; and one of the first things one can study about a linear map is its rank.
\begin{definition}
We say an element $\xi\in T_{B,0}$ has rank $k\in \mathbb{N}$ when the cup-product followed by contraction map $H^0(C,\omega_C)\rightarrow H^1(C,\mathcal{O}_C)$ induced by $\rho(\xi)$ has rank $k$.
\end{definition}
Studying ranks of elements in $T_{B,0}$ for some family $B$ often requires one to reinterpret the infinitesimal period map in a certain way. For example, when the base space $B$ is either the coarse moduli space of smooth curves of genus $g$ ($M_g$) or the base of the Kuranishi family, the Kodaira-Spencer map $\rho: T_{[C]}B \rightarrow H^1(C, T_C)$ is an isomorphism. In this case, in (pp. 271-275 of \cite{griffiths1983}), using (\ref{symmetry}), Griffiths  was able to regard (\ref{infperiod}) as the dual of the multiplication map
\begin{equation}
    \operatorname{Sym}^2 H^0(C,\omega_C)\rightarrow H^0(C,\omega^2_C)
\end{equation}
which is surjective when $C$ is non-hyperelliptic of genus at least 3 (p. 117 of \cite{arbarello1985}). In doing so, he identified the determinantal variety of rank one elements as the intersection of the quadrics containing the canonical curve. In the case when $C$ does not have a $g^1_2,g^1_3$ or $g^2_5$, every rank one element in $H^1(C,T_C)$ is the Schiffer variation at some $p\in C$, the image of the coboundary $H^0(C,T_C(p)|_p)\rightarrow H^1(C,T_C)$ induced by the exact sequence
\[0\rightarrow T_C\rightarrow T_C(p)\rightarrow T_C(p)|_p\rightarrow 0\]
As a result, for generic curve of genus at least 5, all elements of rank one in $H^1(C,T_C)$ are Schiffer Variations and a generic Torelli theorem is established(p. 256 of\cite{Voisin2002} and p. 275 of \cite{griffiths1983}). For generic curve of genus 4, the determinantal variety of rank one is the quadric containing the canonical curve and consists exactly of Schiffer variations and split deformations of rank one in the sense of (p. 6 of\cite{colombo2024} and p. 556 of \cite{raviolo2014}).\\
For a special family $\pi:\mathcal{C}\rightarrow (B,0)$, the study of elements of small rank in $T_{B,0}$ has many implications. By studying rank one deformations that preserve the trigonal structure of a trigonal curve, Beorchia, Pirola and Zucconi were able to conclude that the only $2g-1$ dimensional sub locus of $M_g$ such that, for every general element $[C]$, the Jacobian $J(C)$ is dominated by a Jacobian of hyperelliptic curve of genus $g'\ge g\ge 6$ is the hyperelliptic locus (p. 8869 of\cite{beorchia2020}). Also, by studying the lower bound of rank of deformations that preserve the planarity of plane curves, Favale, Naranjo and Pirola were able to prove Xiao's conjecture for plane curves (pp. 194-196 of \cite{favale2018}). Further relationships with the Coleman-Oort conjecture can be found in (p. 1 of  \cite{colombo2024}) .\\
On the other hand, it was shown in (Lemma 1.3 of \cite{FavalePirola2022} and Lemma 2.4 of\cite{lee2016} and Lemma 2.2 of \cite{gonzalezalonso2021}) that a general element in $H^1(C,T_C)$ has rank $g=\dim H^0(C,\omega_C)$ and Codogni, Gonz\'{a}lez-Alonso and Torelli used this fact to deduce that the Torelli map $\mathcal{M}_g\rightarrow A_g$ is infinitesimally rigid as morphism of DM-Stacks (Theorem 1.1 of \cite{codogni2023}). A natural question then arises: For a special family $\pi:\mathcal{C}\rightarrow (B,0)$, can $T_{B,0}$ contain an element of rank $g=g(C)$? In (Prop. 1.7 of \cite{FavalePirola2022}), Favale and Pirola showed that for a family of curves obtained as a linear system of ample divisors on a surface $S$, the maximal rank of this family can possibly attain is bounded above by $g-q(S)$ where $q(S)$ is the irregularity of the surface due to the fact that the image of $H^0(S,\Omega_S^1)\rightarrow H^0(C,\omega_C)$ is annihilated by every element in this family.\\
Then a subsequent question naturally emerges: Can this upper bound $g-q(S)$ actually be attained by some curve $C$ when $C$ moves within a linear system of sufficiently ample divisors on a surface $S$? In (Theorem 3.1 of \cite{FavalePirola2022}), Favale and Pirola showed that it is true for every smooth curves of degree $d$ in $\mathbb{P}^2$ by realizing infinitesimal variation of Hodge Structure as multiplication maps of Jacobian rings and studying its duality properties. Later on, Gonz\'{a}lez-Alonso and Torelli used generalized Jacobian ring of Green \cite{green_period_map_1985} to develop a general approach to tackle this problem and used this approach to show that this is true for $\mathbb{P}^1\times \mathbb{P}^1$ (Theorem 1.4 of\cite{TorelliGonzalezAlonso2024}) and they formulated the following conjecture
\begin{conjecture}[Gonz\'{a}lez-Alonso-Torelli]Let $S$ be a smooth projective surface. Then, for any smooth curves in any sufficiently ample line bundle $L$ (in the sense of Green\cite{green_period_map_1985}), its deformation within this surface can have rank equal to $g-q(S)$ where $g$ is the genus of the curve and $q(S)$ is the irregularity of the surface.
\end{conjecture}
In this note, we study infinitesimal variation of Hodge structure of nondegenerate curves (See definition~\ref{nondegenerate}) in an ample linear system on smooth toric surface using toric Jacobian ring. This toric Jacobian ring is a modified version of the usual Jacobian ring because the ideal generated by partial derivatives of a hypersurface would no longer be complete intersection in the toric setting(pp.332-333 of \cite{BatyrevCox1994}). Our motivation is as follows: generic canonical trigonal curves of odd genus live on $Bl_p\mathbb{P}^2$ or the Hirzebruch surface $\mathcal{H}_1$. The map to the base $\mathbb{P}^1$ induces their trigonal structure; thus, when one of such curves varies within a linear system of the surface, we can interpret this family as a deformation that preserves the trigonal structure. Motivated by the preceding conjecture, we want to determine whether there exists a trigonal curve on this surface with a trigonal deformation of rank $g$. Furthermore, the moduli space of trigonal curves, $T_g$, can be stratified by the Maroni invariant. This invariant is an integer $r \in \mathbb{Z}_{\ge 0}$ that corresponds to the Hirzebruch surface $\mathcal{H}_r$ containing the canonical trigonal curve. We are also interested in exploring to what extent this method holds as $r$ grows.\\
Unfortunately, adaptation to toric setting is quite subtle. This strategy will not work for arbitrary smooth curve as we need a stronger condition than smoothness to hold, that is nondegeneracy. This condition is a natural condition in studying hypersurfaces in toric varieties. Without this assumption, even for smooth ones, we won't have a good notion of toric residue for toric Jacobian ring (p.73 of \cite{Cox95}) let along any possible dualities between these toric Jacobian rings(p.332 of\cite{BatyrevCox1994} and p.96 of\cite{Loyola2024}). Nonetheless, this nondegeneracy is also a generic assumption and we would have very nice properties among toric Jacobian rings with this assumption. As a result, we may adapt the strategy employed in \cite{FavalePirola2022} and \cite{TorelliGonzalezAlonso2024} to our toric setting; that is we use Castelnuovo Uniform Position Theorem to bound the dimension of some decomposable elements contained in some graded pieces of toric Jacobian ideal. In particular, we have the following:
\begin{Lthm}{\label{A}}For any integer $r \ge 0$, let $\mathcal{H}_r$ denote the corresponding Hirzebruch surface. Suppose $f \in H^0(\mathcal{H}_r, \mathcal{O}_{\mathcal{H}_r}(\beta))$ is a nondegenerate section (Definition~\ref{nondegenerate}) such that the divisors $\beta$ and $\beta+K_{\mathcal{H}_r}$ are ample. Then the curve $C_f = (f=0)$ admits a first-order deformation on $\mathcal{H}_r$ of rank $g$, provided that $\dim J_1(f)_{\beta} \le 9$ (see Definition~\ref{def:ideal}).
\end{Lthm}
To illustrate the usage of the Theorem~\ref{A}, we write out the equation of a nondegenerate $f$ explicitly and compute the dimension of $J_1(f)_{\beta}$ directly. This enables us to establish:
\begin{Lcor}{\label{B}}
    Let $D_1$ and $D_2$ denote the fiber and exceptional divisors, respectively, on the surface $\mathcal{H}_1$ (or $Bl_p\mathbb{P}^2$). For any integer $d\ge 5$, there exists a curve $C_f = (f=0)$ defined by a nondegenerate section $f$ of $dD_1+3D_2$ (see Definition~\ref{nondegenerate}) such that $C_f$ has a first-order deformation on $\mathcal{H}_1$ with rank equal to its genus $g=2d-5$.
\end{Lcor}
We are unable to compute the dimensions of $J_1(f)_{\beta}$ for all nondegenerate curves because when Picard number of the surface is at least 2 we don't know whether these dimensions are the same for all $f$. Nonetheless, $\dim J_1(f)_{\beta}$ is computable using Macaulay2\cite{M2} once an explicit equation is provided.\\
Also, it is possible that our approach works for 2-dimensional toric orbifolds like weighted projective surface $\mathbb{P}(1,1,n)$ or $\mathbb{P}^1\times \mathbb{P}^1/\mathbb{Z}_2$ because both vanishing theorems we used and Griffiths residue are valid for simplicial toric varieties(See theorem~\ref{thm:demazure} and theorem~\ref{thm:Griffiths1969} below). For example, non-hyperelliptic curve of genus 4 is the intersection with a unique quadric and a cubic surface. Generically, the quadric is smooth. However, there is a 8-dimensional family of curve whose unique quadric is singular quadric of rank 3 which is $\mathbb{P}(1,1,2)$. Moreover, the Picard rank of $\mathbb{P}(1,1,2)$ is 1 so we know exactly what $\dim J_1(f)_{\beta}$ is.\\
After the initial submission of this preprint, Sernesi \cite{sernesi2025ivhsnodalplanecurves} established the strongest result so far that a family of curves of genus $g \ge 1$ admitting a degree $d \ge 2$ morphism to a curve $B$ of genus $g'$ has maximal variation if and only if $g'=0$. This approach relies on deformation theory and the study of Severi varieties.

\subsection*{This paper is organized as follows:} Section~\ref{sec:prelim} covers preliminaries of toric geometry, Hodge theory and the set up of the problem. Section~\ref{sec:criterion} establishes a criterion to check whether deformation of rank $g$ exists for a nondegenerate curve on an arbitrary complete and smooth toric surface. We will study this criterion in the case of Hirzebruch surface and prove Theorem~\ref{A} in section~\ref{sec:Hirzebruch}. Lastly, we will do relevant computation for an example in section~\ref{sec:computation} establishing Corollary~\ref{B}.

\subsection*{Notation and Conventions}
Throughout this paper, unless otherwise stated, all varieties are assumed to be smooth, irreducible, and defined over the complex numbers $\mathbb{C}$.
\section*{Acknowledgments}I would like to thank my advisor Prof. Elham Izadi for her mentorship and support, Prof. David Cox for helpful conversations and the anonymous referee for very careful and useful feedbacks.
\section{Review of Toric Geometry of Smooth Toric Surface and Hodge Theory of its ample divisor}
\label{sec:prelim}
\subsection{Basic Construction}
Let $X_{\Sigma}$ be a complete smooth toric surface over $\mathbb{C}$ associated to a fan $\Sigma$ in $N_{\R}=\R^2$ where $N$ is its one-parameter subgroup lattice. Let $M$ be its character lattice and $T_N$ be the associated torus. There is an action on $X_{\Sigma}$ by $T_N$.\\ Denote $\Sigma(n)$ the subset of cones in $\Sigma$ that has dimension $n$. Each $\rho\in \Sigma(1)$ then corresponds to a ray in $\mathbb{R}^2$ emanating from the origin. Take a generator $u=(a,b)$ of this ray. We have $\lim_{t\rightarrow 0}\lambda^{u}(t)=\lim_{t\rightarrow 0}(1,t^a,t^b)=\gamma_{\rho}$ for some unique element $\gamma_{\rho}\in U_{\rho}=\operatorname{Spec}\mathbb{C}[\rho^{\vee}\cap M]\subset X_{\Sigma}$. 
\begin{theorem} Let $X_{\Sigma}$ be a complete smooth toric surface over $\mathbb{C}$ with fan $\Sigma$ in $N_{\mathbb{R}}$. Then$\colon$
\begin{enumerate}
    \item There is an one to one correspondence between torus invariant prime divisors of $X_{\Sigma}$ and $D_{\rho}=\overline{O(\rho)}$ where $O(\rho)$ is the orbit of $\gamma_{\rho}$ under the action of $T_N$ and $\overline{O(\rho)}$ is its Zariski closure in $X_{\Sigma}$. 
    \item Let $u_{\rho}\in \rho\cap N$ be the minimal generator. For $m\in M$
\[\nu_{D_{\rho}}(\chi^m)=\langle m,u_{\rho}\rangle\]
where $\nu_{D_{\rho}}$ is the valuation for $D_{\rho}$.
\item For $m\in M$, let $\operatorname{div}(\chi^m)=\sum_{\rho\in \Sigma(1)}\langle m,u_{\rho}\rangle D_{\rho}$. In this way, we may identify $M$ with a subgroup of $\operatorname{Div}_{T_N}(X_{\Sigma})=\bigoplus_{\rho\in \Sigma(1)}\mathbb{Z}D_{\rho}$, the divisors invariant under the torus action $T_N$. The classical divisor class group $\operatorname{Cl}(X_{\Sigma})$ can be obtained as 
    \[0\rightarrow M\rightarrow \operatorname{Div}_{T_N}(X_{\Sigma})\rightarrow \operatorname{Cl}(X_{\Sigma})\rightarrow 0  \]
    Since $X_{\Sigma}$ is a smooth surface, we even have $\operatorname{Pic}(X_{\Sigma})=\operatorname{Cl}(X_{\Sigma})$
    \item Let $D$ be a $T_N$-invariant Weil divisor on $X_{\Sigma}$, then 
    \[\Gamma(X_{\Sigma},\mathcal{O}_{X_{\Sigma}}(D))=\bigoplus_{\operatorname{div}(\chi^m)+D\ge 0}\mathbb{C}\cdot \chi_m\]
    Since $X_{\Sigma}$ is complete, this is finite dimensional.
    \item the canonical sheaf $\omega_{X_{\Sigma}}\simeq O_{X_{\Sigma}}(-\sum_{\rho\in \Sigma(1)}D_{\rho})$. Thus $K_{X_{\Sigma}}=-\sum_{\rho\in \Sigma(1)}D_{\rho}$ is a torus invariant canonical divisor on $X_{\Sigma}$.
\end{enumerate}
\end{theorem}
\begin{proof}
    (1) follows from theorem 3.2.6 on p.119 of \cite{cox2011toric}. (2) follows from prop. 4.1.1 on p.171 of \cite{cox2011toric}. (3) follows from theorem 4.1.2 and 4.1.3 of pp.171-172 of \cite{cox2011toric}. (4) follows from prop. 4.3.2 of p.189 of \cite{cox2011toric}. (5) follows from theorem 8.2.3 of p.366 of \cite{cox2011toric}.
\end{proof}
\subsection{Cox Ring of a smooth projective toric surface}
\label{subsec:cox}
For each $D_{\rho}$ for $\rho\in \Sigma(1)$, we introduce variable $x_{\rho}$. We have a polynomial ring $S=\mathbb{C}[x_{\rho}:\rho\in \Sigma(1)]$. Each monomial $\prod_{\rho\in \Sigma(1)}x^{a_{\rho}}_{\rho}$ determines a divisor $D=\sum_{\rho\in \Sigma(1)}a_{\rho}D_{\rho}$. Occasionally, we will denote $\prod_{\rho\in \Sigma(1)}x^{a_{\rho}}_{\rho}$ as $x^D$. We will give a grading for $S$ by $\deg(x^D)=[D]\in \operatorname{Pic}(X_{\Sigma})$. For $\alpha\in \operatorname{Pic}(X_{\Sigma})$, we will denote 
\[S_{\alpha}=\bigoplus_{\deg (x^D)=\alpha}\mathbb{C}\cdot x^D\]
Then, we have 
\[S=\mathbb{C}[x_{\rho}:\rho\in \Sigma(1)]=\bigoplus_{\alpha\in \operatorname{Pic}(X_{\Sigma})}S_{\alpha}\]
Also, note that $S_{\alpha}\cdot S_{\beta}\subset S_{\alpha+\beta}$. This is the Cox ring of $X_{\Sigma}$.
\begin{theorem}[Cox]\label{thm:coxring}
    For a smooth toric variety $X_{\Sigma}$. (i) If $\alpha=[D]\in \operatorname{Pic}(X_{\Sigma})$, then there is an isomorphism
    \[\phi_D:S_{\alpha}\simeq H^0(X_{\Sigma},\mathcal{O}_{X_{\Sigma}}(D))\]
    (ii) If $\alpha=[D]$ and $\beta=[E]$, then there is a commutative diagram
    \[\begin{tikzcd}
S_{\alpha}\otimes S_{\beta} \arrow[r] \arrow[d]                    & S_{\alpha+\beta} \arrow[d]  \\
{H^0(X_{\Sigma},\mathcal{O}_{X_{\Sigma}}(D))\otimes H^0(X_{\Sigma},\mathcal{O}_{X_{\Sigma}}(E))} \arrow[r] & {H^0(X_{\Sigma},\mathcal{O}_{X_{\Sigma}}(D+E))}
\end{tikzcd}\]
where the top arrows is multiplication, the bottom arrow is tensor product, and the vertical arrows are the isomorphisms $\phi_D\otimes \phi_E$ and $\phi_{D+E}$.
\end{theorem}
\begin{proof}
    See p.19 of \cite{Cox1995}.
\end{proof}
\begin{theorem}[Generalized Euler Identity]
Suppose that we have complex numbers $\{\phi_{\rho}\}_{\rho\in \Sigma(1)}$ such that $\sum_{\rho\in \Sigma(1)}\phi_{\rho}\cdot \rho=0\in N_{\mathbb{C}}$. Then, for each class $\beta\in \operatorname{Pic}(X_{\Sigma})$, there is a constant $\phi(\beta)$ such that for all $f\in S_{\beta}$, we have 
\[\phi(\beta)f=\sum_{\rho\in \Sigma(1)}\phi_{\rho}x_{\rho}\frac{\partial f}{\partial x_{\rho}}\]
\label{thm:euler}
\end{theorem}
\begin{proof}
    See p.304 of \cite{BatyrevCox1994}
\end{proof}
\subsection{Two Toric Jacobian Rings} We mainly refer to \cite{BatyrevCox1994} and \cite{Loyola2024}. All proofs of the theorems can be found there.\\
In \cite{BatyrevCox1994}, Batyrev and Cox introduced two types of Jacobian rings for a hypersurface in a simplicial toric variety. The first one follows from the standard notion of Jacobian ring with grading in terms of $Cl(X_{\Sigma})$ the only twist. However, this Jacobian ring doesn't have nice properties like being a complete intersection ring or having an isomorphic trace map. In order to emulate the strategies employed in \cite{FavalePirola2022} and \cite{TorelliGonzalezAlonso2024}, we need some nice properties with the Jacobian ring. Fortunately, there is a second notion of Jacobian ring introduced by them. This ring in contrary enjoys some nice properties but we need to have some generic assumptions on the hypersurface. Also, these two rings coincide for weighted projective space but in other cases (when the Picard group has rank at least 2) they are quite different. We will mainly work with the second Jacobian ring. However, we will introduce both of them. 
\begin{definition}
\label{nondegenerate}
Let $f\in S_{\beta}$ be a nonzero polynomial with $\beta$ ample. We say $f$ is nondegenerate if for any $\tau\in \Sigma$, the affine hypersurface $X\cap O(\tau)$ is a smooth subvariety of codimension 1 in $O(\tau)$ or empty where $O(\tau)$ is the orbit of $\gamma_{\tau}$ under the torus action and $\gamma_{\tau}\in U_{\tau}=\operatorname{Spec}[\tau^{\vee}\cap M]\subset X_{\Sigma}$ is the distinguished point.
\end{definition}
\begin{theorem}\label{coxnondegenerate}
    $f$ is nondegenerate if and only if $\{x_{\sigma}\frac{\partial f}{\partial x_{\sigma}}\}_{\sigma\in \Sigma(1)}$ have no common zero on the open set $\mathbb{C}^n-Z(\Sigma)$ where $X_{\Sigma}\cong (\mathbb{C}^n-Z(\Sigma))/G$ is the quotient construction of the toric variety. (See section~\ref{sec:Hirzebruch} for details)
\end{theorem}
\begin{proof}
    See p. 84 of \cite{Cox95}
\end{proof}
\begin{example}
Consider the example given in p.84 of \cite{Cox95}. Let $X=\mathbb{P}^1_{x,y}\times \mathbb{P}^1_{z,w}$ and $f=x^2z^2+x^2w^2+y^2z^2+y^2w^2+\lambda xyzw\in \mathbb{C}[x,y;z,w]$. We have $f$ is smooth if and only if $\lambda\neq \pm 4$ and $f$ is nondegenerate if and only if $\lambda \neq 0,\pm 4$.    
\end{example}
\begin{definition}
    Let $f\in S_{\beta}$ be a nonzero nondegenerate polynomial. Then the toric Jacobian ideal $J(f)\subset S$ is the ideal of $S$ generated by the partial derivatives $\{\frac{\partial f}{\partial x_{\rho}}\}_{\rho\in \Sigma(1)}$. Also, the toric Jacobian ring is $R(f)=S/J(f)$. The grading on $R(f)$ inherits the one on $S$. Notice that, $\frac{\partial f}{\partial x_{\rho}}\in S_{\beta-[D_{\rho}]}$. Also, $J(f)_{\alpha}=J(f)\cap S_{\alpha}$ and $R(f)_{\alpha}=S_{\alpha}/J_{\alpha}$ for $\alpha\in \operatorname{Pic}(X_{\Sigma})$.
\end{definition}
\begin{definition}\label{def:ideal}
Given ideal $J_0(f)=\langle x_{\rho}\partial f/\partial x_{\rho}\rangle_{\rho\in \Sigma(1)}\subset S$, we consider the ideal quotient $J_1(f)=J_0(f):\langle\prod_{\rho\in \Sigma(1)}x_{\rho}\rangle=\{f\in S| f\cdot\prod_{\rho\in \Sigma(1)}x_{\rho}\in J_0(f) \}$. Then, define the second toric Jacobian ring by $R_1(f)=S/J_1(f)$, where $R_1(f)_{\alpha}=S_{\alpha}/J_1(f)_{\alpha}$ and $J_1(f)_{\alpha}=J_1(f)\cap S_{\alpha}$. We will similarly denote $R_0(f)=S/J_0(f)$.
\end{definition}
\begin{remark}
Notice that we have $J(f)\subset J_1(f)$. In particular, we have a surjection $R(f)\twoheadrightarrow R_1(f)$ inducing $R(f)_{\alpha}\twoheadrightarrow R_1(f)_{\alpha}$ for each $\alpha\in \operatorname{Pic}(X_{\Sigma})$ but in general they are not the same. This will be illustrated below. 
\end{remark}
\begin{theorem}[Batyrev\&Cox] If $Y\subset X_{\Sigma}$ is a nondegenerate ample hypersurface defined by $f\in S_{\beta}$, then there is a natural isomorphism 
\[H^{p,1-p}(Y)_{van}\cong R_1(f)_{(2-p)\beta+K_{X_{\Sigma}}}\]
\label{thm:BCox}
\end{theorem}
\begin{proof}
    See p.332 of \cite{BatyrevCox1994}
\end{proof}
Now, since $H^1(X_{\Sigma},\mathbb{C})=0$, we have 
\[H^{p,1-p}(Y)=H^{p,1-p}(Y)_{van}\cong R_1(f)_{(2-p)\beta+K_{X_{\Sigma}}}\]
In particular, we have 
\[H^0(Y,\omega_Y)=H^{1,0}(Y)\cong R_1(f)_{\beta+K_{X_{\Sigma}}}\]
\[H^1(Y,\mathcal{O}_Y)=H^{0,1}(Y)\cong R_1(f)_{2\beta+K_{X_{\Sigma}}}\]
\begin{remark}
    In general, we have $H^0(Y,\omega_Y)\cong R(f)_{\beta+K_{X_{\Sigma}}}$ but for $H^1(Y,\mathcal{O}_Y)$,we have the exact sequence $0\rightarrow H^0(X_{\Sigma})\xrightarrow{\cup [Y]}H^2(X_{\Sigma})\rightarrow R(f)_{2\beta+K_{X_{\Sigma}}}\rightarrow H^1(Y,\mathcal{O}_Y)\rightarrow 0$
\end{remark}
\begin{remark}
    We have $H^0(Y,\omega_Y)=H^0(X_{\Sigma},\beta+K_{\Sigma})=S_{\beta+K_{\Sigma}}\Longrightarrow S_{\beta+K_{\Sigma}}=R_1(f)_{\beta+K_{\Sigma}}$. This can be seen from twisting the ideal sheaf sequence of $Y$ in $X$ by $\beta+K_{\Sigma}$ and $h^{2,0}(X_{\Sigma})=h^{2,1}(X_{\Sigma})=0$.
    \label{remark:genus}
\end{remark}
\begin{theorem}[Cox]
    For $f\in S_{\beta}$ nondegenerate and ample, there is a toric residue map $\operatorname{Res}\colon R_0(f)_{3\beta+K_{X_{\Sigma}}}\rightarrow \mathbb{C}$ such that the composition 
    \[R_1(f)_{3\beta+2K_{X_{\Sigma}}}\xrightarrow{\cdot \prod_{\rho\in \Sigma(1)}x_{\rho}}R_0(f)_{3\beta+K_{X_{\Sigma}}}\xrightarrow{\operatorname{Res}} \mathbb{C}\] is an isomorphism. We will also call this map $\operatorname{Res}$. 
\end{theorem}
\begin{proof}
    See p. 85 of \cite{Cox95}
\end{proof}
\begin{theorem}[Villaflor Loyola]
For $f\in S_{\beta}$ nondegenerate and ample, in the assumption that $\prod_{\rho\in \Sigma(1)}x_{\rho}\not\in J_0(f)$, then for any $s\in \mathbb{Z}$ and $t\in \{-1,0,1,2\}$, we have the multiplication map 
\[R_1(f)_{s\beta+tK_{X_{\Sigma}}}\times R_1(f)_{(3-s)\beta+(2-t)K_{X_{\Sigma}}}\rightarrow R_1(f)_{3\beta+2K_{X_{\Sigma}}}\xrightarrow{\cong}\mathbb{C} \]
induces a injection 
\[R_1(f)_{s\beta+tK_{X_{\Sigma}}}\hookrightarrow (R_1(f)_{(3-s)\beta+(2-t)K_{X_{\Sigma}}})^{\vee}\]
In particular, we have perfect pairings for 
\[R_1(f)_{\beta+K_{X_{\Sigma}}}\times R_1(f)_{2\beta+K_{X_{\Sigma}}}\rightarrow R_1(f)_{3\beta+2K_{X_{\Sigma}}}\xrightarrow{\cong}\mathbb{C}\]
\[R_1(f)_{\beta}\times R_1(f)_{2\beta+2K_{X_{\Sigma}}}\rightarrow R_1(f)_{3\beta+2K_{X_{\Sigma}}}\xrightarrow{\cong}\mathbb{C}\]
\label{thm:loyola}
\end{theorem}
\begin{proof}
    See Corollary 3.1 of \cite{Loyola2024}
\end{proof}
\begin{remark}
    This assumption $\prod_{\rho\in \Sigma(1)}x_{\rho}\not\in J_0(f)$ is automatically satisfied for curve with genus at least 1 on complete smooth toric surfaces. See Remark 3.2 of \cite{Loyola2024}.
    \label{remark:asum}
\end{remark}
Now, we present a lemma that is analogous to a lemma proved in \cite{FavalePirola2022} using the same method. Villaflor Loyola established a more general one in \cite{Loyola2024} but we include this one for notational purpose.\\
For $\alpha\neq 0\in S_{\beta+K}=R_1(f)_{\beta+K_{\Sigma}}$, denote $K_{\gamma}(\alpha)$ to be the kernel of the map $R_1(f)_{\gamma}\xrightarrow{\cdot \alpha}R_1(f)_{\gamma+\beta+{K_{\Sigma}}}$.
\begin{lemma}
    $R_1(f)_{2\beta+K_{\Sigma}}/K_{2\beta+K_{\Sigma}}(\alpha)\cong \mathbb{C}$ and the pairing 
    \[R_1(f)_{\beta}/K_{\beta}(\alpha)\times R_1(f)_{\beta+K_{\Sigma}}/K_{\beta+K_{\Sigma}}(\alpha)\rightarrow R_1(f)_{2\beta+K_{\Sigma}}/K_{2\beta+K_{\Sigma}}(\alpha)\cong \mathbb{C}\]
    is perfect inducing 
    \[\dim R_1(f)_{\beta}/K_{\beta}(\alpha)=\dim R_1(f)_{\beta+K_{\Sigma}}/K_{\beta+K_{\Sigma}}(\alpha)\]
    \label{lemma:duality}
\end{lemma}
\begin{proof}
Since $\alpha\neq 0$ and by Theorem~\ref{thm:loyola}  we have a perfect pairing 
\[R_1(f)_{\beta+K_{X_{\Sigma}}}\times R_1(f)_{2\beta+K_{X_{\Sigma}}}\rightarrow R_1(f)_{3\beta+2K_{X_{\Sigma}}}\xrightarrow{\cong}\mathbb{C},\]
we know that
\[R_1(f)_{2\beta+K_{\Sigma}}\xrightarrow{\cdot \alpha}R_1(f)_{3\beta+2K_{\Sigma}}\] is surjective. Hence, we have an isomorphism 
\[R_1(f)_{2\beta+K_{\Sigma}}/K_{2\beta+K_{\Sigma}}(\alpha)\cong \mathbb{C}\]
Now, we will show that the pairing
\[R_1(f)_{\beta}/K_{\beta}(\alpha)\times R_1(f)_{\beta+K_{\Sigma}}/K_{\beta+K_{\Sigma}}(\alpha)\rightarrow R_1(f)_{2\beta+K_{\Sigma}}/K_{2\beta+K_{\Sigma}}(\alpha)\cong \mathbb{C}\]
is perfect.\\
Take $[x]\neq 0\in R_1(f)_{\beta}/K_{\beta}(\alpha)$, we have $x\alpha\neq 0\in R_1(f)_{2\beta+K_{\Sigma}}$. By the perfect pairing above, we have $y\in R_1(f)_{\beta+K_{\Sigma}}$ such that $x\alpha y\neq 0\in R_1(f)_{3\beta+2K_{\Sigma}}$. Therefore, we have $[x]\cdot [y]=[xy]\neq 0\in R_1(f)_{2\beta+K_{\Sigma}}/K_{2\beta+K_{\Sigma}}(\alpha)\cong \mathbb{C}$ for some $[y]\neq 0\in R_1(f)_{\beta+K_{\Sigma}}/K_{\beta+K_{\Sigma}}(\alpha)$. Now, take $[x']\neq 0\in R_1(f)_{\beta+K_{\Sigma}}/K_{\beta+K_{\Sigma}}(\alpha)$, a similar argument shows that we may find $[y']\neq 0\in R_1(f)_{\beta}/K_{\beta}(\alpha)$ such that $[x']\cdot [y']\neq 0\in R_1(f)_{2\beta+K_{\Sigma}}/K_{2\beta+K_{\Sigma}}(\alpha)$. Therefore, we have this pairing is perfect in particular, we have 
\[\dim R_1(f)_{\beta}/K_{\beta}(\alpha)=\dim R_1(f)_{\beta+K_{\Sigma}}/K_{\beta+K_{\Sigma}}(\alpha)\]
\end{proof}
\subsection{IVHS of Ample Sections in Toric Surface} We mainly refer to Chapter 6 of \cite{Voisin2003} and \cite{BatyrevCox1994}
\begin{theorem}[Demazure-Bott-Steenbrink-Danilov Vanishing]
For a complete simplicial toric variety $X$, let $\mathcal{L}$ be an ample invertible sheaf. Then, for $p\ge 0$ and $i>0$, we have 
\[H^i(X,\Omega^p_X\otimes \mathcal{L})=0\]
\label{thm:demazure}
\end{theorem}
\begin{proof}
    See Theorem 7.1 of \cite{BatyrevCox1994}
\end{proof}
In particular, this satisfies the condition on p.160 in section 6.1.2 of \cite{Voisin2003}. Notice that $X$ doesn't need to be smooth.
\begin{theorem}[Griffiths 1969]
    Let $X$ be a complete simplicial toric variety of dimension $n$ and $Y$ be a generic smooth section of $\mathcal{L}$ as above. Denote $U=X\backslash Y$. For every integer $p$ between $1$ and $n$, the image of the natural map 
    \[H^0(X,K_X(pY))\rightarrow H^n(U,\mathbb{C})\]
    which to a section $\alpha$ (viewed as a meromorphic form on $X$ of degree $n$, and therefore closed, holomorphic on $U$ and having a pole of order $p$ along $Y$) associates its de Rham cohomology class, is equal to $F^{n-p+1}H^n(U)$. Via Griffiths residue, $F^{n-p+1}H^n(U)=F^{n-p}H^{n-1}(Y,\mathbb{C})_{van}$. 
    \label{thm:Griffiths1969}
\end{theorem}
\begin{proof}
    See p. 160 of \cite{Voisin2003}
\end{proof}
Let $B\subset |\beta|=\mathbb{P}(S_{\beta})$ be the open subset parametrizing smooth nondegenerate curves in the fixed linear system $\beta$ on the smooth toric surface $X_{\Sigma}$ where $\beta$ is ample. We have a universal smooth curve $\pi:\mathcal{Y}\rightarrow B.$ For $y\in B$, denote $Y=\pi^{-1}(y)$.\\
By Theorem~\ref{thm:Griffiths1969}, we have two maps
\[\gamma_1\colon H^0(X_{\Sigma},K_{X_{\Sigma}}(Y))\rightarrow H^{1,0}(Y)_{van}\cong H^{1,0}(Y) \]
\[\gamma_2\colon H^0(X_{\Sigma},K_{X_{\Sigma}}(2Y))\rightarrow H^{0,1}(Y)_{van}\cong H^{0,1}(Y) \]
The infinitesimal variation of Hodge structure of $Y$ in $X_{\Sigma}$ is the map 
\[\overline{\nabla}:H^{1,0}(Y)\rightarrow \operatorname{Hom}(T_{B,y},H^{0,1}(Y))\]
where the map is induced by the Gauss-Manin connection.
\begin{theorem}[Carlson\&Griffiths (1980)]
The map \[\overline{\nabla}:H^{1,0}(Y)\rightarrow \operatorname{Hom}(T_{B,y},H^{0,1}(Y))\]
can be described as follows. For $P\in H^0(X_{\Sigma},K_{X_{\Sigma}}(Y))$ and $H\in T_{B,y}$, we have 
\[\overline{\nabla}(\gamma_1(P))(H)=c\gamma_2(PH)\]
where $PH$ is the image of $H^0(X_{\Sigma},K_{X_{\Sigma}}(Y))\times H^0(X_{\Sigma},O(Y)) \rightarrow H^0(X_{\Sigma},K_{X_{\Sigma}}(2Y))$ and $c$ is some constant.
\label{thm:CarlsonGriffiths1980}
\end{theorem}
\begin{proof}
    See p.167 of \cite{Voisin2003}
\end{proof}
Together with Theorem~\ref{thm:coxring},Theorem~\ref{thm:BCox},Theorem~\ref{thm:Griffiths1969},Theorem~\ref{thm:CarlsonGriffiths1980} and that the tangent space of $f$ in $B\subset |\beta|$ can be identified with $\Gamma(X_{\Sigma},\beta)/(f)=S_{\beta}/(f)$ (We don't take isomorphism into account), we have 
\begin{corollary}
The map\[\overline{\nabla}:H^{1,0}(Y)\rightarrow \operatorname{Hom}(T_{B,f},H^{0,1}(Y))\]agrees, up to a non-zero scalar multiple, with the map induced by multiplication:\[R_1(f)_{\beta+K_{\Sigma}}\rightarrow \operatorname{Hom}(S_{\beta}/(f),R_1(f)_{2\beta+K_{\Sigma}}).\]Therefore, the corresponding map\[T_{B,f}\rightarrow \operatorname{Hom}(H^0(Y,\omega_Y),H^1(Y,\mathcal{O}_Y))\]is likewise identified, up to a scalar multiple, with the multiplication map\[S_{\beta}/(f)\rightarrow \operatorname{Hom}(R_1(f)_{\beta+K_{\Sigma}},R_1(f)_{2\beta+K_{\Sigma}}).\]
\end{corollary}
Our goal in this note is to show that for $f$ nondegenerate and ample there exists an element in $S_{\beta}/(f)$ such that the morphism in $\operatorname{Hom}(R_1(f)_{\beta+K_{\Sigma}},R_1(f)_{2\beta+K_{\Sigma}})$ given by the product of the element is an isomorphism. By Theorem~\ref{thm:euler}, we know that $f\in J_1(f)_{\beta}$ so $S_{\beta}/(f)$ surjects onto $R_1(f)_{\beta}$. As a result, it suffices to show that there exists an element in $R_1(f)_{\beta}$ such that the morphism in $\operatorname{Hom}(R_1(f)_{\beta+K_{\Sigma}},R_1(f)_{2\beta+K_{\Sigma}})$ given by the multiplication of the element is an isomorphism because the lift of that element in $S_{\beta}/(f)$ will be the one we want.\\
\textbf{Goal}: For non-degenerate $f$, there exists an element in $R_1(f)_{\beta}$ such that the morphism in $\operatorname{Hom}(R_1(f)_{\beta+K_{\Sigma}},R_1(f)_{2\beta+K_{\Sigma}})$ induced by multiplication of that element is an isomorphism. 
\section{A Criterion}
\label{sec:criterion}
We prove an analogue of a lemma in \cite{FavalePirola2022} and \cite{TorelliGonzalezAlonso2024} using essentially the same strategy. We will make the assumption that $\dim S_{\beta+K_{\Sigma}}\neq 0$ otherwise the curve is rational and the claim is trivial. This assumption also implies that $\prod_{\rho\in \Sigma(1)}x_{\rho}\not\in J_0(f)$ so that we can use Theorem~\ref{thm:loyola}(See Remark~\ref{remark:asum}). We will also use the same notation as in Lemma~\ref{lemma:duality}.
\begin{lemma}\label{lemma:square} Take $\beta\in \operatorname{Pic}(X_{\Sigma})$ ample and $f\in S_{\beta}$ a nondegenerate element. For any $\eta\in R_1(f)_{\beta}$ such that the rank of the map $R_1(f)_{\beta+K_{\Sigma}}\xrightarrow{\cdot \eta} R_1(f)_{2\beta+K_{\Sigma}}$ is maximal, if $\eta\cdot \alpha=0\in R_1(f)_{2\beta+K_{\Sigma}}$ for $\alpha\in R_1(f)_{\beta+K_{\Sigma}}$, then $\alpha^2=0\in R_1(f)_{2\beta+2K_{\Sigma}}$.
\end{lemma}
\begin{proof}
Take $\eta\in R_1(f)_{\beta}$ such that $R_1(f)_{\beta+K_{\Sigma}}\xrightarrow{\cdot \eta} R_1(f)_{2\beta+K_{\Sigma}}$ has maximal rank and $\alpha\in R_1(f)_{\beta+K_{\Sigma}}$ such that $\eta\cdot \alpha=0\in R_1(f)_{2\beta+K_{\Sigma}}$. Without loss of generality, we may assume $\alpha\neq 0\in R_1(f)_{\beta+K_{\Sigma}}$. First, we will show that $\eta'\alpha^2=0\in R_1(f)_{3\beta+2K_{\Sigma}}\cong \mathbb{C}$ for any $\eta'\in R_1(f)_{\beta}$. This is obviously true for $\eta$. Take any other $\eta'\in R_1(f)_{\beta}$. If $\eta'\alpha=0$, there is nothing to prove. Therefore, we may assume that $\eta'\alpha\neq 0\in R_1(f)_{2\beta+K_{\Sigma}}$. Consider $h_t=\eta+t\eta'$. Since $\cdot \eta$ has maximal rank and having maximal rank is an open condition, $\cdot h_t$ has maximal rank for all $t$ in a small analytic disk $\mathbb{D}$. The kernels of the maps \[R_1(f)_{\beta+K_{\Sigma}}\xrightarrow{\cdot h_t}R_1(f)_{2\beta+K_{\Sigma}}\]
have constant dimension and form a holomorphic vector subbundle on this small disk $\mathbb{D}$. In particular, we may choose a path $\gamma(t)\in R_1(f)_{\beta+K_{\Sigma}}$ such that $\gamma(0)=\alpha$ and $\gamma(t)\cdot (\eta+t\eta')=0$ for all $t$. Then for $n\ge 1$ the expansion of $\gamma(t)$ is of the form 
\[\gamma(t)=\alpha+\alpha't^n+O(t^{n+1})\]
Then, we have 
\[(\alpha+\alpha't^n+O(t^{n+1}))\eta+(\alpha+\alpha't^n+O(t^{n+1}))t\eta'\equiv 0\]
Since $\eta\cdot \alpha=0$, we have 
\[\eta'\alpha t+t^n\eta \alpha'+O(t^{n+1})\equiv 0\]
Since $\eta'\alpha\neq 0$, we have the only possibility for this to hold is when $n=1$ and $\eta'\alpha+\eta\alpha'=0$. Hence, we have $(\eta'\alpha+\eta\alpha')\alpha=0$ and $\eta'\alpha^2=0\in R_1(f)_{3\beta+2K_{\Sigma}}\cong \mathbb{C}$. 

This means that the map induced by multiplication of $\alpha^2$
\[R_1(f)_{\beta}\xrightarrow{\cdot \alpha^2}R_1(f)_{3\beta+2K_{\Sigma}}\cong \mathbb{C}\]
is the zero map.
By Theorem~\ref{thm:loyola}, there is an injection $R_1(f)_{2\beta+2K_{X_{\Sigma}}}\hookrightarrow (R_1(f)_{\beta})^{\vee}$ that sends an element in $R_1(f)_{2\beta+2K_{X_{\Sigma}}}$ to the map it induced in $(R_1(f)_{\beta})^{\vee}$ given by multiplication. Since $\alpha^2$ and $0$ induce the same map, $\alpha^2=0$ as elements in $R_1(f)_{2\beta+2K_{\Sigma}}$.
\end{proof}
To prove the main theorem, it suffices to show that 
\[\alpha^2=0\in R_1(f)_{2\beta+2K_{\Sigma}}\Longrightarrow \alpha=0\in R_1(f)_{\beta+K_{\Sigma}}\]
For $a\in V$, we denote $[a]\in \mathbb{P}(V)$ the image of $a$. For the above reason, we consider the morphism 
\[s:\mathbb{P}(S_{\beta+K_{\Sigma}})\times \mathbb{P}(S_{\beta+K_{\Sigma}})\rightarrow \mathbb{P}(S_{2\beta+2K_{\Sigma}})\]
where $s([\alpha],[\beta])=[\alpha\beta]\in \mathbb{P}(S_{2\beta+2K_{\Sigma}})$. 

We will denote $D$ to be the image of this morphism $s$. Note that $\mathbb{P}(J_1(f)_{2\beta+2K_{\Sigma}})\subset \mathbb{P}(S_{2\beta+2K_{\Sigma}})$. We will denote $D^{2\beta+2K_{\Sigma}}_J$ as the intersection of $D$ with $\mathbb{P}(J_1(f)_{2\beta+2K_{\Sigma}})$.\\
Also, recall the notation in Lemma~\ref{lemma:duality}, for $\alpha\neq 0\in S_{\beta+K_{\Sigma}}=R_1(f)_{\beta+K_{\Sigma}}$, denote $K_{\gamma}(\alpha)$ to be the kernel of the map $R_1(f)_{\gamma}\xrightarrow{\cdot \alpha}R_1(f)_{\gamma+\beta+{K_{\Sigma}}}$.
\begin{theorem}
    Let $\beta$ be an ample class such that $J_1(f)_{2\beta+2K_{\Sigma}}\neq 0$. Let $C$ be the zero locus of a non-degenerate $f\in S_{\beta}$ and $g$ be the genus of $C$. If $\dim (D^{2\beta+2K_{\Sigma}}_J)<g-1$, the rank of $R_1(f)_{\beta+K_{\Sigma}}\xrightarrow{\cdot \eta} R_1(f)_{2\beta+K_{\Sigma}}$ is equal to $g=\dim R_1(f)_{\beta+K_{\Sigma}}=h^{1,0}(C)$ for some $\eta\in R_1(f)_{\beta}$. 
    \label{thm:bound}
\end{theorem}
\begin{remark}
    The assumption that $J_1(f)_{2\beta+2K_{\Sigma}}\neq 0$ is satisfied, for example, when $\beta+K_{\Sigma}$ is ample.
\end{remark}
\begin{proof}
     First, we observe that the morphism 
     \[s:\mathbb{P}(S_{\beta+K_{\Sigma}})\times \mathbb{P}(S_{\beta+K_{\Sigma}})\rightarrow \mathbb{P}(S_{2\beta+2K_{\Sigma}})\] is finite. Indeed, via the identification in Theorem~\ref{thm:coxring}, the above morphism can be identified with 
    \[|\beta+K_{\Sigma}|\times |\beta+K_{\Sigma}|\rightarrow |2\beta+2K_{\Sigma}| \]
    that sends a pair of effective divisors $(E_1,E_2)$ where $E_i\in |\beta+K_{\Sigma}|$ for $i=1,2$ to an effective divisor $E_1+E_2\in |2\beta+2K_{\Sigma}|$. Take an effective divisor $E\sim 2\beta+2K_{\Sigma}$ such that $E=E_1+E_2$ for some effective divisor $E_1,E_2\sim \beta+K_{\Sigma}$. We may write $E=\sum^n_{j=1}k_j F_j$, $E_1=\sum^n_{j=1}k_{j,1} F_j$ and $E_2=\sum^n_{j=1}k_{j,2} F_j$ where $F_j$'s are prime divisors of $X_{\Sigma}$, $k_j,k_{j,1},k_{j,2}\in \mathbb{Z}_{\ge 0}$ and $k_j=k_{j,1}+k_{j,2}$. Hence, the number of pairs $(E_1,E_2)$ such that $E_1+E_2=E$ is finite. Since $s$ is a morphism between projective varieties, it is a proper morphism. It has finite fibers so it is a finite morphism. 

    We will now proceed to prove our Theorem. For the sake of contradiction, assume that for any $\eta\in R_1(f)_{\beta}$, the rank $R_1(f)_{\beta+K_{\Sigma}}\xrightarrow{\cdot \eta} R_1(f)_{2\beta+K_{\Sigma}}$ is strictly less than $g$. We will show that this implies that $\dim (D^{2\beta+2K_{\Sigma}}_J)\ge g-1$.\\
    By Remark~\ref{remark:genus}, we have $S_{\beta+K_{\Sigma}}=R_1(f)_{\beta+K_{\Sigma}}$
    Consider \[Y=\{[\alpha]\in \mathbb{P}(S_{\beta+K_{\Sigma}})|\alpha^2\in J_1(f)_{2\beta+2K_{\Sigma}} \}=\{[\alpha]\in \mathbb{P}(R_1(f)_{\beta+K_{\Sigma}})|\alpha^2=0\in R_1(f)_{2\beta+2K_{\Sigma}} \}\]
    and the incidence correspondence 
    \[\tilde{I}=\{([\eta],[\alpha])\in \mathbb{P}(S_{\beta})\times \mathbb{P}(S_{\beta+K_{\Sigma}})|\eta\alpha\in J_1(f)_{2\beta+K_{\Sigma}}\}\]
    with projections $p_1$ and $p_2$.\\
    Consider the map 
    \[S_{\beta}\xrightarrow{\cdot \alpha}R_1(f)_{2\beta+K_{\Sigma}} \]Denote $W_{\beta}(\alpha)$ to be its kernel. We have 
    \[\mathbb{P}(W_{\beta}(\alpha))=\{[\eta]\in \mathbb{P}(S_{\beta})|\eta\alpha\in J_1(f)_{2\beta+K_{\Sigma}}\}\]
    It is easy to see that $S_{\beta}/W_{\beta}(\alpha)\cong R_1(f)_{\beta}/K_{\beta}(\alpha)$.\\
    By our assumption, we have $p_1:\tilde{I}\rightarrow \mathbb{P}(S_{\beta})$ is surjective. We have an irreducible component $I$ of $\tilde{I}$ that dominates $\mathbb{P}(S_{\beta})$. Hence,
    \[\dim (I)\ge \dim \mathbb{P}(S_{\beta}).\]
    Let $U$ be the open dense subset of $I$ consisting of pairs $([\eta],[\alpha])$ such that $R_1(f)_{\beta+K_{\Sigma}}\xrightarrow{\cdot \eta} R_1(f)_{2\beta+K_{\Sigma}}$ has maximal rank. Then, we have $\alpha^2\in J_1(f)_{2\beta+2K_{\Sigma}}$ for $[\alpha]$ in the pair $([\eta],[\alpha])$ by Lemma~\ref{lemma:square}. This implies that $p_2(U)\subset Y$. Then, we have $\dim (Y)\ge \dim (p_2(U))=\dim (p_2(I))$. We also have, for $[\alpha]\in p_2(I)$
    \[\dim (p_2^{-1}([\alpha]))+\dim Y\ge \dim (p_2^{-1}([\alpha]))+\dim (p_2(I))\ge \dim I\]
    and $\dim (p_2^{-1}([\alpha])=\dim \mathbb{P}(W_{\beta}(\alpha))$. Therefore, we have 
    \[\dim Y\ge \dim I-\dim \mathbb{P}(W_{\beta}(\alpha))\ge \dim \mathbb{P}(S_{\beta})-\dim \mathbb{P}(W_{\beta}(\alpha)) \]
    By lemma~\ref{lemma:duality}, we have 
    \[\dim \mathbb{P}(S_{\beta})-\dim \mathbb{P}(W_{\beta}(\alpha))=\dim \mathbb{P}(S_{\beta}/W_{\beta}(\alpha))+1=\dim \mathbb{P}(R_1(f)_{\beta}/K_{\beta}(\alpha))+1\]
    \[=\dim \mathbb{P}(R_1(f)_{\beta+K_{\Sigma}}/K_{\beta+K_{\Sigma}}(\alpha))+1=\dim R_1(f)_{\beta+K_{\Sigma}}-\dim K_{\beta+K_{\Sigma}}(\alpha)=g-\dim K_{\beta+K_{\Sigma}}(\alpha)\]
    This implies that 
    \[\dim Y\ge g-\dim K_{\beta+K_{\Sigma}}(\alpha)\]
    Now, consider \[Z=\{([\alpha],[\beta])\in Y\times \mathbb{P}(S_{\beta+K_{\Sigma}})|\alpha\beta\in J_1(f)_{2\beta+2K_{\Sigma}}\}\]
    It has a projection map $\pi_1$ to $Y$. Notice that $\pi_1$ is surjective with $([\alpha],[\alpha])\in \pi_1^{-1}([\alpha])=[\alpha]\times \mathbb{P}(K_{\beta+K_{\Sigma}}(\alpha))$ for any $[\alpha]\in Y$. Also, the restriction of the morphism $s$ to $Z$ has image in $D^{2\beta+2K_{\Sigma}}_J$. We will denote this map as $\psi$.
    Let $Z'$ be an irreducible component of $Z$ that dominants $Y$. For a generic $[\alpha]\in Y$, by fiber-dimension theorem,
    \[\dim \mathbb{P}(K_{\beta+K_{\Sigma}}(\alpha))=\dim \pi_1^{-1}([\alpha])=\dim Z'-\dim Y\]
    Since the morphsim $s$ is finite hence so is $\psi$, we have 
    \[\dim D^{2\beta+2K_{\Sigma}}_J\ge \dim \psi(Z')=\dim Z'= \dim (Y)+\dim \mathbb{P}(K_{\beta+K_{\Sigma}}(\alpha))\]
    \[=\dim Y+\dim K_{\beta+K_{\Sigma}}(\alpha)-1\ge g-1\]
    This is what we want.
    \end{proof}
\section{Hirzebruch Surface}
\label{sec:Hirzebruch}
Now, we will study the case for Hirzebruch surface. In particular, we want to bound $\dim D_J^{2\beta+2K_{\Sigma}}$ when $X_{\Sigma}=\mathcal{H}_r$ for $r\ge0$. The main reference is \cite{cox2011toric}
\subsection{GIT Quotient} All complete simplicial toric varieties have a geometric quotient construction by the group $G=\operatorname{Hom}(\text{Cl}(X_{\Sigma}),\mathbb{C}^*)\subset (\mathbb{C}^*)^{\Sigma(1)}$ (See Chapter 5 of \cite{cox2011toric}). The one for Hirzebruch surface is the following:\\
For $r\ge 0$, the ray generators of the fan for a Hirzebruch surface $\mathcal{H}_r$ are $u_1=-e_1+re_2, u_2=e_2,u_3=e_1,u_4=-e_2$, where $e_i$ are standard basis in $N\otimes \mathbb{R}=\mathbb{R}^2$.\\
We have $(t_1,\cdots, t_4)\in (\mathbb{C}^*)^{4}$ lies in $G$ if and only if 
\[t_1^{\langle m,-e_1+re_2 \rangle}t_2^{\langle m,e_2 \rangle}t_3^{\langle m,e_1\rangle}t_4^{\langle m,-e_2 \rangle}=1\]
When $m=e_1$,
\[t_1^{-1}t_3=1\Longrightarrow t_3=t_1\]
When $m=e_2$,
\[t_1^rt_2t_4^{-1}=1\Longrightarrow t_4=t_1^rt_2\]
Then, we have 
\[G=\{(\lambda,\mu,\lambda,\lambda^r\mu)|\lambda,\mu\in \mathbb{C}^*\}\simeq (\mathbb{C}^*)^2\]
For $\sigma_1=\text{Cone}(-e_1+re_2, e_2)$, $\sigma_2=\text{Cone}(e_1,e_2)$, $\sigma_3=\text{Cone}(e_1,-e_2)$ and $\sigma_4=\text{Cone}(-e_2,-e_1+re_2)$, if we use the coordinate of Cox ring $S=\mathbb{C}[x_1,x_2,x_3,x_4]$ introduced in section~\ref{subsec:cox} where $x_i$ corresponds to the ray $u_i$, we have the exceptional set is given by $Z(\Sigma)=V(B(\Sigma))$ where $B(\Sigma)=\{x^{\hat{\sigma_1}}=x_3x_4,x^{\hat{\sigma_2}}=x_1x_4,x^{\hat{\sigma_3}}=x_1x_2,x^{\hat{\sigma_4}}=x_2x_3\}$.\\
Then, we have the geometric quotient $(\mathbb{C}^4-Z(\Sigma))\sslash G$ with the action $(\lambda,\mu)\cdot(a,b,c,d)= (\lambda a,\mu b,\lambda c,\lambda^r\mu d)$. Also, $(\mathbb{C}^4-Z(\Sigma))\sslash G$ is isomorphic to $\mathcal{H}_r$. This is the quotient construction for Hirzebruch surface.
\subsection{Picard Group} By Orbit-Cone correspondence, there are divisors $D_i$ corresponding to the rays $u_i$ for $i=1,2,3,4$ with the relations
\[0\sim \operatorname{div}(\chi^{e_1})=\langle e_1,u_1\rangle D_1+\langle e_1,u_2\rangle D_2+\langle e_1,u_3\rangle D_3+\langle e_1,u_4\rangle D_4=-D_1+D_3\]
\[0\sim \operatorname{div}(\chi^{e_2})=\langle e_2,u_1\rangle D_1+\langle e_2,u_2\rangle D_2+\langle e_2,u_3\rangle D_3+\langle e_2,u_4\rangle D_4=rD_1+D_2-D_4\]
In particular, we have $D_1=D_3$ and $rD_1+D_2=D_4$ in $\operatorname{Pic}(\mathcal{H}_r)$. This implies that $\operatorname{Pic}(\mathcal{H}_r)=\mathbb{Z}^2$ is the generated by $D_1$ and $D_2$.\\
We also know that $aD_1+bD_2$ is ample and very ample if and only if $a>rb$ and $b>0$. There are also the intersection relations
\[D_1\cdot D_1=0,\quad D_2\cdot D_2=-r,\quad D_3\cdot D_3=0,\quad D_4\cdot D_4=r\]
\[D_1\cdot D_2=D_2\cdot D_3=D_3\cdot D_4=D_4\cdot D_1=1.\]
The canonical divisor is given by 
\[K_{\mathcal{H}_r}=-D_1-D_2-D_3-D_4=-(r+2)D_1-2D_2\]
Lastly, by Noether formula (Special case of Hirzebruch Riemann Roch. See Chapter 5 of \cite{huybrechts2004complex}) and Demazure Vanishing (Theorem~\ref{thm:demazure}), for $a>rb$ and $b>0$, we have 
\[h^0(\mathcal{H}_r,aD_1+bD_2)=\chi(\mathcal{H}_r,aD_1+bD_2)=\frac{(aD_1+bD_2)\cdot (aD_1+bD_2-K_{\mathcal{H}_r})}{2}+\chi(\mathcal{H}_r,\mathcal{O}_{\mathcal{H}_r})\]
\[=(a+1)(b+1)-\frac{rb(b+1)}{2}\]
Also, the genus of the curve given by a section of $aD_1+bD_2$  is 
\[1+\frac{1}{2}((aD_1+bD_2)^2+(aD_1+bD_2)\cdot K_{\mathcal{H}_r} )\]
\[=(b-1)(a-1-\frac{rb}{2})\]
\subsection{A Linear System}
For an ample divisor $\beta=a D_1+bD_2$ where $a>rb$ and $b>0$ and a nondegenerate section $f\in \Gamma(\mathcal{H}_r,a D_1+bD_2)$, from Definition~\ref{def:ideal}, we can consider $J_1(f)_{\beta}\subset S_{\beta}\cong \Gamma(\mathcal{H}_r,a D_1+bD_2)$.\\
By Theorem~\ref{thm:euler} (the Generalized Euler Identity), since $u_2+u_4=0$, $u_1+u_3-ru_2=0$ and $u_1+u_3+ru_4=0$, there are numbers $A,B,C\neq 0$ such that
\begin{align*}
    Af&=x_2\frac{\partial f}{\partial x_2}+x_4\frac{\partial f}{\partial x_4}\\
    Bf&=x_1\frac{\partial f}{\partial x_1}-rx_2\frac{\partial f}{\partial x_2}+x_3\frac{\partial f}{\partial x_3}
\end{align*}
\begin{theorem}\label{thm:birational}The linear system $J_1(f)_{\beta}$ restricted to the curve $C=\{f=0\}$ is base point free and the morphism induced by it is birational onto its image.
\end{theorem}
\begin{proof}
First, the linear system $J_1(f)_{\beta}$ contains $\{x_1\frac{\partial f}{\partial x_1},x_2\frac{\partial f}{\partial x_2},x_3\frac{\partial f}{\partial x_3},x_4\frac{\partial f}{\partial x_4} \}$. The nondegeneracy of $f$ (Theorem~\ref{coxnondegenerate} implies that these sections share no common zeros outside the exceptional locus $Z(\Sigma)$. Because points of the smooth curve $C$ corresponds to orbits in $\mathbb{C}^{\Sigma(1)} \setminus Z(\Sigma)$, the linear system is base-point free.\\
Now, we will show that the morphism it induces is birational onto its image. It suffices to show that the map is generically injective onto its image. Consider points in the open set $W=C\cap \{\frac{\partial f}{\partial x_i}\neq 0\}_{i=1,2,3,4}$. We will show that $W$ is nonempty and dense.

If $W=\varnothing$, then for $p\in C$, we have $\frac{\partial f}{\partial x_i}(p)=0$ for some $i\in \{1,2,3,4\}$. Since $C$ is irreducible, we must have $\frac{\partial f}{\partial x_i}(p)=0$ for all $p\in C$ for a fixed $i$. This shows that $V(f)\subset V(\frac{\partial f}{\partial x_i})$. However, this is not by weak Toric Nullstellensatz and degree reasons (pp.221-222 of \cite{cox2011toric}). By the same reasoning, $W$ must be dense; indeed, $C$ is irreducible so $C\cap V(\frac{\partial f}{\partial x_i})$ must be finite set of points if $C\cap V(\frac{\partial f}{\partial x_i})$ is not equal to $C$.

From the generalized Euler Identity, for $p\in C$, we have 
\begin{align}
    x_2(p)\frac{\partial f}{\partial x_2}(p)+x_4(p)\frac{\partial f}{\partial x_4}(p)&=0\label{eq:t1}\\
    x_1(p)\frac{\partial f}{\partial x_1}(p)-rx_2(p)\frac{\partial f}{\partial x_2}(p)+x_3(p)\frac{\partial f}{\partial x_3}(p)&=0\label{eq:t2}
\end{align}
For $p\in W=C\cap \{\frac{\partial f}{\partial x_i}\neq 0\}_{i=1,2,3,4}$, we have $x_2(p)\neq 0$. Indeed, if $x_2(p)=0$, by equation~\ref{eq:t1}, we must have $x_4(p)=0$ resulting $p$ to be in the exceptional set. Also, since $x_1(p)$ and $x_3(p)$ cannot be both zero (or else the point $p$ is in the exceptional set), for $p\in W$, we either have $x_1(p),x_2(p)\neq 0$ or $x_2(p),x_3(p)\neq 0$. In the quotient construction, we have $p=(a,b,c,d)\sim (\lambda a,\mu b, \lambda c,\lambda^r\mu d)$ for $\lambda,\mu\in \mathbb{C}^*$. Therefore, we have $W$ is covered by the open chart $U_{12}=\{x_1=x_2=1\}$ and $U_{23}=\{x_2=x_3=1\}$.

Now, for point $p\in W\cap U_{12}$, by considering 
    \[\{x_1(p)\frac{\partial f}{\partial x_1}(p),x_1(p)\frac{\partial f}{\partial x_3}(p),x_2(p)\frac{\partial f}{\partial x_2}(p), x_1^{r}(p)x_2(p)\frac{\partial f}{\partial x_4}(p)\},\] since $x_1(p)=x_2(p)=1$,
    \[x_4(p)=-(x_2(p)\frac{\partial f}{\partial x_2}(p))/(\frac{\partial f}{\partial x_4}(p))=-(x_2(p)\frac{\partial f}{\partial x_2}(p))/(x_1^r(p)x_2(p)\frac{\partial f}{\partial x_4}(p))\]
    Also, we have 
    \[x_3(p)=(rx_2(p)\frac{\partial f}{\partial x_2}(p)-x_1(p)\frac{\partial f}{\partial x_1}(p))/(\frac{\partial f}{\partial x_3}(p))\]
    \[=(rx_2(p)\frac{\partial f}{\partial x_2}(p)-x_1(p)\frac{\partial f}{\partial x_1}(p))/(x_1(p)\frac{\partial f}{\partial x_3}(p))\]
    Hence, we may conclude the morphism on $W\cap U_{12}$ is injective.
    
    The case for $p\in W\cap U_{23}$ is the same. Hence, we may conclude that the morphism induced by the linear system of $J_1(f)_{\beta}$
    is birational onto its image.
\end{proof}
\begin{remark}
    When $f$ is just smooth and not nondegenerate, $J_1(f)_{\beta}$ will have base points. We can resolve this issue by consider for example $J_1(f)_{\beta+rD_1}$. However, since dualities of toric Jacobian ring would not hold without non-degeneracy assumption, we stress that we can only consider the case when $f$ is nondegenerate. 
\end{remark}
\begin{remark}
    Since we have two generalized Euler identities, when restricted to the curve, the image of this morphism will be in $\mathbb{P}^{\dim J_1(f)_{\beta}-2}$
\end{remark}
\subsection{Estimate}
Now, we will use the same approach used in \cite{FavalePirola2022} and \cite{TorelliGonzalezAlonso2024} to calculate a bound for $D_J^{2\beta+2K_{\mathcal{H}_r}}$. Recall that $D$ is defined to be the image of the morphism $s:\mathbb{P}(S_{\beta+K_{\Sigma}})\times \mathbb{P}(S_{\beta+K_{\Sigma}})\rightarrow \mathbb{P}(S_{2\beta+2K_{\Sigma}})$; and $D^{2\beta+2K_{\Sigma}}_J$ is the intersection of $D$ with $\mathbb{P}(J_1(f)_{2\beta+2K_{\Sigma}})\subset \mathbb{P}(S_{2\beta+2K_{\Sigma}})$. From Theorem~\ref{thm:bound}, we want to establish that $\dim D_J^{2\beta+2K_{\mathcal{H}_r}}<g-1$.

Under the assumption that $J_1(f)_{2\beta+2K_{\mathcal{H}_r}}\neq 0$, take $g\in J_1(f)_{\beta}\backslash\{\mathbb{C}f\}$, consider the linear space
\[E_g=\{sf+tg|s,t\in S_{\beta+2K_{\mathcal{H}_r}}\}\]
It is obvious that $E_g\subset J_1(f)_{2\beta+2K_{\mathcal{H}_r}}$. 
\begin{lemma} Under the assumption that $J_1(f)_{2\beta+2K_{\mathcal{H}_r}}\neq 0$. For any $g\in J_1(f)_{\beta}\backslash \{\mathbb{C}f\}$, $sf+tg\neq 0$ in $S_{2\beta+2K_{\mathcal{H}_r}}$ for any $s,t\in S_{\beta+2K_{\mathcal{H}_r}}$. In particular, $\dim E_g=2\dim S_{\beta+2K_{\mathcal{H}_r}}$.
\end{lemma}
\begin{proof}
Suppose we have $sf+tg=0$. Then, we may restrict it to the curve $Z=\{f=0\}$. Then, for $p\in Z$, we have $t(p)g(p)=0$. Then, $(V(g)\cap Z)\cup (V(t)\cap Z)=Z$. Since $g$ is not a multiple of $f$ by assumption, we have $V(t)\cap Z=Z$ by irreducibility of $Z$. Then, we have $t$ is a multiple of $f$. Since $t\in S_{\beta+2K_{\mathcal{H}_r}}=S_{\beta-[(2r+4)D_1+4D_2]}$ and $f\in S_{\beta}$, this is not possible unless $t=0$. In particular, we have $sf=0\in S_{2\beta+2K_{\mathcal{H}_r}}$. Since $f\neq 0$, we have $s=0$ as the Cox ring is a integral domain. 
\end{proof}
\begin{theorem}[Castelnuovo Uniform Position Theorem] The monodromy group of the universal hyperplane section of an irreducible curve $C\subset \mathbb{P}^r$ is the full symmetric group $S_n$ where $n$ is the number of points in the hyperplane sections.
\end{theorem}
\begin{proof}
    For a beautiful and detailed background and proof, see Chapter 11 of \cite{MR4898511}.
\end{proof}
\begin{theorem}
    Under the assumption that $J_1(f)_{2\beta+2K_{\mathcal{H}_r}}\neq 0$, for a general $g\in J_1(f)_{\beta}\backslash \{\mathbb{C}f\}$, $D_{J}^{2\beta+2K_{\mathcal{H}_r}}\cap \mathbb{P}(E_g)=\varnothing$. Hence, $\dim D_J^{2\beta+2K_{\mathcal{H}_r}}\le \dim \mathbb{P}J_1(f)_{2\beta+2K_{\mathcal{H}_r}}-\dim \mathbb{P}(E_g)$. 
    \label{thm:bb}
\end{theorem}
\begin{proof}
Write $\beta=[aD_1+bD_2]\in \operatorname{Pic}(\mathcal{H}_r)$. By Theorem~\ref{thm:birational}, we know that $J_1(f)_{\beta}\subset S_{\beta}\cong\Gamma(\mathcal{H}_r,aD_1+bD_2)$ is base point free and the morphism it induces is birational onto its image. In particular, we have a map $\psi:Z\rightarrow C\subset \mathbb{P}^r$ where $r\le \dim J_1(f)_{\beta}-2$ and $C$ is a curve that is not contained in any hyperplanes of $\mathbb{P}^r$. Take a generic $g\in J_1(f)_{\beta}\backslash \{\mathbb{C}f\}$ such that the support of $\operatorname{div}(g|_Z)$ is reduced and consists of distinct point $p_1+\cdots+p_n$ where $n=\deg \mathcal{O}_Z(aD_1+bD_2)=(aD_1+bD_2)^2$. Since $a>rb>0$ due to ampleness of $aD_1+bD_2$, 
\[n=(aD_1+bD_2)^2=2ab-rb^2> 2(b-1)(a-1-\frac{rb}{2})-2=2g(Z)-2.\]The support of $\operatorname{div}(g|_Z)$ is given by a pullback of hyperplane section of $\mathbb{P}^r$. Now, consider $U=\{h\in J_1(f)_{\beta}\backslash \{\mathbb{C} f\}| (h|_Z)\text{ is reduced }\}$. By Castelnuovo Uniform Position Theorem above, we know that the induced monodromy action by the fundamental group $\pi_1(U,g)\rightarrow \operatorname{Sym}(\{p_1,\cdots, p_n\})$ is surjective. For the sake of contradiction, assume the theorem doesn't hold. Then, there exist $s,t\in S_{\beta+2K_{\mathcal{H}_r}}$ such that $sf+tg=AB$ for some $A,B\in S_{\beta+K_{\mathcal{H}_r}}\backslash\{0\}$ and $AB\in  J_1(f)_{2\beta+2K_{\mathcal{H}_r}}$. When restricted to $Z$, we have 
\[\operatorname{div}(A|_Z)+\operatorname{div}(B|_Z)=\operatorname{div}(t|_Z)+\operatorname{div}(g|_Z)=\operatorname{div}(t|_Z)+p_1+\cdots+p_n.\]Without loss of generality, we may assume $\operatorname{div}(A|_Z)$ contains $p_1,\cdots, p_m$ where $m\ge \frac{n}{2}$ but doesn't contain $p_{m+1},\cdots, p_n$. Now, for $p_i$ where $1\le i\le m$, we know there is a $\rho\in \pi_1(U,g)$ whose induced action permutes $p_i$ and $p_{m+1}$. In particular, we may construct $A_i\in S_{\beta+K_{\mathcal{H}_r}}\cong \Gamma(\mathcal{H}_r,\beta+K_{\mathcal{H}_r})$ such that there are $s_i,t_i\in S_{\beta+2K_{\mathcal{H}_r}}$ and $B_i\in S_{\beta+K_{\mathcal{H}_r}}$ with $s_if+t_i g=A_iB_i$ and $A_i$ vanishes at $p_{m+1}$ and $p_j$ where $1\le j\le m$ and $j\neq i$ but not vanishing at $p_i$ (for a detailed construction see part (2) of Theorem 4.4 of \cite{TorelliGonzalezAlonso2024}). By adjunction formula, $A_i$'s restricted to $Z$ are holomorphic 1-forms on $Z$. Suppose we have $c_{m+1}A|_Z+\sum_{k=1}^m c_kA_k|_Z=0$ for constants $c_i\in \mathbb{C}$. Then, for $1\le i\le m$, we have $0=c_{m+1}A|_Z(p_i)+\sum_{k=1}^m c_kA_k|_Z(p_i)=c_iA_i|_Z(p_i)$. Since $A_i(p_i)\neq 0$, we have $c_i=0$ for $1\le i\le m$. Also, $0=c_{m+1}A|_Z(p_{m+1})+\sum_{k=1}^m c_kA_k|_Z(p_{m+1})=c_{m+1}A|_Z(p_{m+1})$ implies that $c_{m+1}=0$ as $A|_Z(p_{m+1})\neq 0$. Therefore, we have constructed $m+1$ linearly independent 1-forms. This implies that $m+1\le g(Z)$. However, $2m\ge n> 2g(Z)-2$ so we have a contradiction. 
\end{proof}
\subsection{Proof of Main Theorem}
We will now proceed to prove Theorem~\ref{A}.
\begin{proof}[Proof of Theorem \ref{A}]
Note that the assumption $J_1(f)_{2\beta+2K_{\mathcal{H}_r}}\neq 0$ is automatically satisfied for $\beta+K_{\mathcal{H}_r}$ ample.\\
By Theorem~\ref{thm:bound}, it suffices to establish $\dim D_J^{2\beta+2K_{\Sigma}}<g-1$.\\
By Theorem~\ref{thm:bb},
\[\dim D_J^{2\beta+2K_{\mathcal{H}_r}}\le  \dim \mathbb{P}J_1(f)_{2\beta+2K_{\mathcal{H}_r}}-\dim \mathbb{P}(E_g) \]
\[=\dim J_1(f)_{2\beta+2K_{\mathcal{H}_r}}-1-(\dim E_g-1) \]
\[=(\dim S_{2\beta+2K_{\mathcal{H}_r}}-\dim R_1(f)_{2\beta+2K_{\mathcal{H}_r}})-(2\dim S_{\beta+2K_{\mathcal{H}_r}})\]
By Theorem~\ref{thm:loyola}, we know that $\dim R_1(f)_{2\beta+2K_{\mathcal{H}_r}}=\dim R_1(f)_{\beta}$.
\[\dim R_1(f)_{\beta}=\dim S_{\beta}-\dim J_1(f)_{\beta}\]
We have 
\[\dim D_J^{2\beta+2K_{\mathcal{H}_r}}\le \dim J_1(f)_{\beta}+\dim S_{2\beta+2K_{\mathcal{H}_r}}-\dim S_{\beta}-2\dim S_{\beta+2K_{\mathcal{H}_r}}\]
Since $\beta+K_{\mathcal{H}_r}$ is ample, $h^i(H_r,\beta+2K_{H_r})=0$ for $i=1,2$ by Kodaira Vanishing Theorem. By Hirzebruch-Riemann-Roch, 
\[2\dim {S_{\beta+2K_{\mathcal{H}_r}}}=2h^0(\mathcal{H}_r,\beta+2K_{H_r})=2\chi(\beta+2K_{H_r})\]
\[=2(\frac{1}{2}(\beta+2K_{H_r})(\beta+K_{H_r})+1)=\beta^2+3\beta\cdot K_{\mathcal{H}_r}+2K_{\mathcal{H}_r}^2+2\]
Write $2\beta+2K_{\mathcal{H}_r}=K_{\mathcal{H}_r}+(\beta+K_{\mathcal{H}_r})+\beta$ where $\beta$ and $\beta+K_{\mathcal{H}_r}$ are ample, by Kodaira Vanishing Theorem and Hirzebruch-Riemann-Roch, 
\[\dim {S_{2\beta+2K_{\mathcal{H}_r}}}=h^0(\mathcal{H}_r,2\beta+2K_{H_r})=\chi(2\beta+2K_{H_r})\]
\[=\frac{1}{2}(2\beta+2K_{H_r})(2\beta+K_{H_r})+1=2\beta^2+3\beta\cdot K_{\mathcal{H}_r}+K_{\mathcal{H}_r}^2+1\]
By Demazure Vanishing (Theorem~\ref{thm:demazure}) and Hirzebruch-Riemann-Roch, 
\[\dim S_{\beta}=h^0(\mathcal{H}_r,\beta)=\chi(\beta)=\frac{1}{2}\beta\cdot (\beta-K_{\mathcal{H}_r})+1=\frac{1}{2}\beta^2-\frac{1}{2}\beta\cdot K_{\mathcal{H}_r}+1\]
Notice that $K^2_{\mathcal{H}_r}=12-n$ where $n$ is the number of two-dimensional cones in $\Sigma$ (See \cite{cox2011toric}), we have $K^2_{\mathcal{H}_r}=8$.
If we put everything together, 
\[\dim S_{2\beta+2K_{\mathcal{H}_r}}-\dim S_{\beta}-2\dim S_{\beta+2K_{\mathcal{H}_r}}\]
\[=2\beta^2+3\beta\cdot K_{\mathcal{H}_r}+K_{\mathcal{H}_r}^2+1-(\frac{1}{2}\beta^2-\frac{1}{2}\beta\cdot K_{\mathcal{H}_r}+1)-(\beta^2+3\beta\cdot K_{\mathcal{H}_r}+2K_{\mathcal{H}_r}^2+2)\]
\[=\frac{1}{2}\beta^2+\frac{1}{2}\beta\cdot K_{\mathcal{H}_r}-K^2_{\mathcal{H}_r}-2\]
\[=\frac{1}{2}\beta^2+\frac{1}{2}\beta\cdot K_{\mathcal{H}_r}+1-11=g-11\]
We have $\dim D_J^{2\beta+2K_{\Sigma}}<g-1$ as long as 
\[\dim J_1(f)_{\beta}+g-11<g -1\]
or 
\[\dim J_1(f)_{\beta}\le 9 \]
\end{proof}

\section{Computation}\label{sec:computation}
In this section, we will give an example of nondegenerate section $f$ of $dD_1+3D_2$ on $\mathcal{H}_1$ and show that $\dim J_1(f)_{\beta}=7$ for $\beta=[dD_1+3D_2]$. Therefore, by Theorem~\ref{A}, this establishes Corollary~\ref{B}. The set up is the following.

Let $\mathcal{H}_1$ be the Hirzebruch surface (or $\mathbb{P}^2$ blown up at one point), viewed as a toric variety. We work in the Cox ring $S = \mathbb{C}[x_1, x_2, x_3, x_4]$. The variables are graded by $\operatorname{Pic}(\mathcal{H}_1)$ according to the divisors $D_1$ (the fiber) and $D_2$ (the exceptional divisor):
\begin{align*}
    \deg(x_1) &= \deg(x_3) = [D_1] \\
    \deg(x_2) &= [D_2] \\
    \deg(x_4) &= [D_1 + D_2]
\end{align*}
Let $f \in S_{\beta}$ be a polynomial of degree $\beta = [dD_1 + 3D_2]$ given by $f=x_1^dx_2^3+x_3^{d-3}x_4^3+x_3^dx_2^3+x_1^{d-3}x_4^3$. We will first verify that $f$ is nondegenerate.
\begin{theorem} The polynomial $f=x_1^dx_2^3+x_3^{d-3}x_4^3+x_3^dx_2^3+x_1^{d-3}x_4^3$ is nondegenerate for $d\ge 5$.
\end{theorem}
\begin{proof}
    By Theorem~\ref{coxnondegenerate}, it suffices to show that $\langle x_1\frac{\partial f}{\partial x_1},x_2\frac{\partial f}{\partial x_2},x_3\frac{\partial f}{\partial x_3},x_4\frac{\partial f}{\partial x_4} \rangle$ have no common zero in $\mathbb{C}^4-Z$ where $Z=V(B(\Sigma))$ and $B(\Sigma)=\{x_3x_4,x_1x_4,x_1x_2,x_2x_3\}$.

    We first write \[ f = P(x_1, x_3)x_2^3 + Q(x_1, x_3)x_4^3 \]
where $P = x_1^d + x_3^d$ and $Q = x_1^{d-3} + x_3^{d-3}$.

Then, 
\begin{align}
    x_2\frac{\partial f}{\partial x_2} &= 3P x_2^3 \label{eq:e2} \\
    x_4\frac{\partial f}{\partial x_4} &= 3Q x_4^3 \label{eq:e4} \\
    x_1\frac{\partial f}{\partial x_1} &= dx_1^d x_2^3 + (d-3)x_1^{d-3}x_4^3 \label{eq:e1} \\
    x_3\frac{\partial f}{\partial x_3} &= dx_3^d x_2^3 + (d-3)x_3^{d-3}x_4^3 \label{eq:e3}
\end{align}
Assume for contradiction there exists a point $p = (x_1, x_2, x_3, x_4) \in \mathbb{C}^4 \setminus Z$ where all four logarithmic derivatives vanish. Because $p \notin Z$, we know $(x_1, x_3) \neq (0,0)$ and $(x_2, x_4) \neq (0,0)$.

From (\ref{eq:e2}) and (\ref{eq:e4}), the vanishing of $x_2\frac{\partial f}{\partial x_2}$ and $x_4\frac{\partial f}{\partial x_4}$ requires that either $x_2 = 0$ or $P(x_1, x_3) = 0$, and either $x_4 = 0$ or $Q(x_1, x_3) = 0$.

Suppose $x_2 = 0$. Since $(x_2, x_4) \neq (0,0)$, we must have $x_4 \neq 0$. Substituting $x_2 = 0$ into (\ref{eq:e1}) and (\ref{eq:e3}) yields:
\[ (d-3)x_1^{d-3}x_4^3 = 0 \quad \text{and} \quad (d-3)x_3^{d-3}x_4^3 = 0 \]
Since $x_4 \neq 0$, this forces $x_1 = 0$ and $x_3 = 0$. This implies $p \in Z$, contradicting our assumption. By symmetry, assuming $x_4 = 0$ forces $x_1 = x_3 = 0$, leading to the same contradiction.

Therefore, we are strictly forced to assume $x_2 \neq 0$ and $x_4 \neq 0$. This requires $P(x_1, x_3) = 0$ and $Q(x_1, x_3) = 0$. 

We must find a common root for $P = x_1^d + x_3^d = 0$ and $Q = x_1^{d-3} + x_3^{d-3} = 0$ where $(x_1, x_3) \neq (0,0)$. If $x_3 = 0$, then $x_1 = 0$, inducing that $p\in Z$. Assuming $x_3 \neq 0$, we set $x_3 = 1$, reducing the system to $x_1^d = -1$ and $x_1^{d-3} = -1$. 

Dividing these equations gives $x_1^3 = 1$. Substituting a third root of unity $x_1 = e^{2\pi i k / 3}$ (for $k \in \mathbb{Z}$) back into $x_1^d = -1$ yields $\frac{2kd}{3} = 2m+1$ for some integer $m$. This implies that $2kd=6m+3$ which is not possible because one is even and the other is odd. Therefore, $\langle x_1\frac{\partial f}{\partial x_1},x_2\frac{\partial f}{\partial x_2},x_3\frac{\partial f}{\partial x_3},x_4\frac{\partial f}{\partial x_4} \rangle$ have no common zero on $\mathbb{C}^4-Z$. This shows that $f$ is non-degenerate. 
\end{proof}
Now, we will compute the dimension of $J_1(f)_{\beta}$ for $\beta=[dD_1+3D_2]$.
\begin{theorem}\label{hand}
    For polynomial $f=x_1^dx_2^3+x_3^{d-3}x_4^3+x_3^dx_2^3+x_1^{d-3}x_4^3\in S_{\beta}$ where $d\ge 5$, $\dim J_1(f)_{\beta}=7$.
\end{theorem}
\begin{proof}
We first write 
\[f=P(x_1,x_3)x_2^3+Q(x_1,x_3)x_4^3\]
where $P=x_1^d+x_3^d$ and $Q=x_1^{d-3}+x_3^{d-3}$.

The Jacobian ideal $J_0(f)$ is generated by 
\begin{align*}
    x_2 \partial_{x_2}f &= 3P x_2^3 \\
    x_4 \partial_{x_4}f &= 3Q x_4^3 \\
    x_1 \partial_{x_1}f &= x_1^d x_2^3 + (d-3)x_1^{d-3} x_4^3 \\
    x_3 \partial_{x_3}f &= dx_3^d x_2^3 + (d-3)x_3^{d-3} x_4^3
\end{align*}
Notice that we have 
\[\frac{d}{3}x_2 \partial_{x_2}f+\frac{d-3}{3}x_4 \partial_{x_4}f=x_1 \partial_{x_1}f+x_3 \partial_{x_3}f\]
One sees that $J_0(f)=\langle x_1\frac{\partial}{\partial x_1}f,x_2\frac{\partial}{\partial x_2}f,x_4\frac{\partial}{\partial x_4}f \rangle$.
Note that $x_1x_2x_3x_4\in S_{\mu}$ where $\mu=[3D_1+2D_2]$ and recall that $J_1(f)=J_0(f)\colon \langle x_1x_2x_3x_4\rangle$. Take $\psi \in J_1(f)_{\beta}$, we have
\[\psi x_1x_2x_3x_4=A(x_1\frac{\partial}{\partial x_1}f)+B(x_2\frac{\partial}{\partial x_2}f)+C(x_4\frac{\partial}{\partial x_4}f)\]
where $A,B,C\in S_{\mu}$.

Since $\psi\in S_{\beta}$ and $\beta=[dD_1+3D_2]$, we may write it as
\[\psi = H_d x_2^3 + H_{d-1} x_2^2 x_4 + H_{d-2} x_2 x_4^2 + H_{d-3} x_4^3 \]
where $H_i\in \mathbb{C}[x_1,x_3]$.\\
Since $\deg(x_1)=\deg(x_3)=[D_1]$, $\deg(x_2)=[D_2]$ and $\deg(x_4)=[D_1+D_2]$, $H_i$ has degree $i$ respectively in $\mathbb{C}[x_1,x_3]$.
\[ \text{LHS} =  (x_1 x_3 H_d) x_2^4 x_4 + (x_1x_3H_{d-1}) x_2^3 x_4^2 + (x_1x_3H_{d-2} )x_2^2 x_4^3 + (x_1x_3H_{d-3}) x_2 x_4^4 \]
Since $A,B,C\in S_{\mu}$ where $\mu=[3D_1+2D_2]$, we may similarly write
\begin{align*}
    A &= a_3 x_2^2 + a_2 x_2 x_4 + a_1 x_4^2 \\
    B &= b_3 x_2^2 + b_2 x_2 x_4 + b_1 x_4^2 \\
    C &= c_3 x_2^2 + c_2 x_2 x_4 + c_1 x_4^2
\end{align*}
where $a_i,b_i,c_i$ have degree $i$ in $\mathbb{C}[x_1,x_3]$ for $i=1,2,3$.
Now, equating the two sides, by looking at the term involving $x_2^ix_4^j$ for $i+j=5$, we have the following six relations in $\mathbb{C}[x_1,x_3]$:
\begin{align}
    a_3x_1^d+3b_3(x_1^d+x_3^d)&=0\label{eq:z1}\\
    a_2x_1^d+3b_2(x_1^d+x_3^d)&=x_1x_3H_d\label{eq:z2}\\
    a_1x_1^d+3b_1(x_1^d+x_3^d)&=x_1x_3H_{d-1}\label{eq:z3}\\
    (d-3)a_3x_1^{d-3}+3c_3(x_1^{d-3}+x_3^{d-3})&=x_1x_3H_{d-2}\label{eq:z4}\\
    (d-3)a_2x_1^{d-3}+3c_2(x_1^{d-3}+x_3^{d-3})&=x_1x_3H_{d-3}\label{eq:z5}\\
    (d-3)a_1x_1^{d-3}+3c_1(x_1^{d-3}+x_3^{d-3})&=0\label{eq:z6}
\end{align}
From equations~\ref{eq:z1} and \ref{eq:z6}, because $d-3\ge 2$ and $a_i,b_i$ and $c_i$ have degree $i$ in $\mathbb{C}[x_1,x_3]$, we have $a_1=c_1=a_3=b_3=0$. This reduces the above equations to  
\begin{align}
    a_2x_1^d+3b_2(x_1^d+x_3^d)&=x_1x_3H_d\label{eq:y2}\\
    3b_1(x_1^d+x_3^d)&=x_1x_3H_{d-1}\label{eq:y3}\\
    3c_3(x_1^{d-3}+x_3^{d-3})&=x_1x_3H_{d-2}\label{eq:y4}\\
    (d-3)a_2x_1^{d-3}+3c_2(x_1^{d-3}+x_3^{d-3})&=x_1x_3H_{d-3}\label{eq:y5}
\end{align}
From equation~\ref{eq:y3}, we see that $x_1x_3$ must divide $b_1$ but $b_1$ has degree $1$ so $b_1=0$ implying $H_{d-1}=0$.\\
From equation~\ref{eq:y4}, we see that $x_1x_3$ must divide $c_3$ so $c_3=x_1x_3(\alpha_1 x_1+\alpha_2 x_3)$ since $c_3$ has degree $3$ where $\alpha_i\in \mathbb{C}$ for $i=1,2$ are two free parameters. Therefore, $H_{d-2}=3(\alpha_1x_1+\alpha_2x_2)(x_1^{d-3}+x_3^{d-3})$.\\
From equations~\ref{eq:y2} and \ref{eq:y5}, we see that $x_1$ must divides $b_2$ and $c_2$ and $x_3$ must divide $a_2+3b_2$ and $(d-3)a_2+3c_2$.\\
Write 
\[a_2=\gamma_1x_1^2+\gamma_2x_1x_3+\gamma_3x_3^2\]
where $\gamma_i\in \mathbb{C}$ for $i=1,2,3$ are free parameters. We see that 
\begin{align*}
    b_2&=\frac{-1}{3}\gamma_1x_1^2+\delta_1x_1x_3\\
    c_2&=-\frac{(d-3)}{3}\gamma_1 x_1^2+\delta_2 x_1x_3
\end{align*}
where $\delta_i$ for $i=1,2$ are also free parameters.
In particular, we have 
\begin{align*}
    H_d&=-\gamma_1x_1x_3^{d-1}+\gamma_2x_1^d+\gamma_3x_1^{d-1}x_3+3\delta_1(x_1^d+x_3^d)\\
    H_{d-1}&=0\\
    H_{d-2}&=3(\alpha_1x_1+\alpha_2x_2)(x_1^{d-3}+x_3^{d-3})\\
    H_{d-3}&=-(d-3)\gamma_1x_1^2x_3^{d-3}+(d-3)\gamma_2x_1^{d-3}+(d-3)\gamma_3x_1^{d-4}x_3+3\delta_2(x_1^{d-3}+x_3^{d-3})
\end{align*}
Notice that we have in total 7 free parameters and they are $\alpha_1,\alpha_2,\gamma_1,\gamma_2,\gamma_3,\delta_1,\delta_2$. This enables us to conclude that $\dim J_1(f)_{\beta}=7$.
\end{proof}
\begin{remark} On a Hirzebruch surface $\mathcal{H}_r$ where $r\ge 1$, take $f$ to be a section of an ample section $\beta$ such that $\beta+K_{\mathcal{H}_r}$ is also ample. $J_1(f)_{\beta}$ contains
\[J(f)_{\beta}=\{a\frac{\partial f}{\partial x_1}+b\frac{\partial f}{\partial x_2}+c\frac{\partial f}{\partial x_3}+d\frac{\partial f}{\partial x_4}|a,c\in S_{[D_1]}, b\in S_{[D_2]},d\in S_{[rD_1+D_2]}\}\]
Now, $S_{[D_1]}=\{x_1,x_3\}$, $S_{[D_2]}=\{x_2\}$ and $S_{[D_4]}=\{x_1^kx_3^{r-k}x_2,x_4\}_{0\le k\le r}$. By the generalized Euler identity (Theorem~\ref{thm:euler}), there is one relation among them, therefore, we have 
\[\dim J(f)_{\beta}=2+2+1+r+2-1=r+6\]
In particular, $\dim J_1(f)_{\beta}\ge r+6$. When the toric surface has Picard rank at least 2, the precise relationship between $J_1(f)_{\beta}$ and $J(f)_{\beta}$ is not known. However, we are aware that using algorithm involving Gr\"{o}bner basis (for example a modified version of the one described in \cite{MR2290010}), one can calculate $\dim J_1(f)_{\beta}$ explicitly by hand or using Macaulay 2\cite{M2} once an explicit equation is provided. 
\end{remark}

\printbibliography

@book{cox2011toric,
  author = {Cox, David A. and Little, John B. and Schenck, Henry K.},
  title = {Toric Varieties},
  year = {2011},
  publisher = {American Mathematical Society},
  address = {Providence, RI},
  doi = {10.1090/gsm/124}
}

@article{Cox1995,
  AUTHOR = {Cox, David A.},
     TITLE = {The homogeneous coordinate ring of a toric variety},
   JOURNAL = {J. Algebraic Geom.},
  FJOURNAL = {Journal of Algebraic Geometry},
    VOLUME = {4},
      YEAR = {1995},
    NUMBER = {1},
     PAGES = {17--50},
      ISSN = {1056-3911,1534-7486},
   MRCLASS = {14M25},
  MRNUMBER = {1299003},
MRREVIEWER = {Mina\ Teicher},
}

@book{Voisin2002,
  author = {Voisin, Claire},
  title = {Hodge Theory and Complex Algebraic Geometry I},
  publisher = {Cambridge University Press},
  year = {2002},
  series = {Cambridge Studies in Advanced Mathematics},
  volume = {76},
  isbn = {9780521802604},
  doi = {10.1017/CBO9780511615344}
}

@book{Voisin2003,
  author = {Voisin, Claire},
  title = {Hodge Theory and Complex Algebraic Geometry II},
  volume = {2},
  series = {Cambridge Studies in Advanced Mathematics},
  number = {77},
  publisher = {Cambridge University Press},
  address = {Cambridge},
  year = {2003},
  isbn = {9780521718028},
  doi = {10.1017/CBO9780511615177},
  url = {https://www.cambridge.org/core/books/hodge-theory-and-complex-algebraic-geometry-ii/47F0D9D694725A426F9CAFDFA830DB11}
}

@article{BatyrevCox1994,
  author = {Batyrev, Victor V. and Cox, David A.},
  title = {On the Hodge Structure of Projective Hypersurfaces in Toric Varieties},
  journal = {Duke Mathematical Journal},
  volume = {75},
  number = {2},
  pages = {293--338},
  year = {1994},
  publisher = {Duke University Press}
}

@article{FavalePirola2022,
  author = {Favale, Filippo Francesco and Pirola, Gian Pietro},
  title = {Infinitesimal Variation Functions for Families of Smooth Varieties},
  journal = {Milan Journal of Mathematics},
  volume = {90},
  number = {2},
  pages = {451--467},
  year = {2022},
  month = {12},
  doi = {10.1007/s00032-022-00353-2},
  url = {https://link.springer.com/article/10.1007/s00032-022-00353-2}
}

@article{TorelliGonzalezAlonso2024,
  author = {Gonz{\'a}lez-Alonso, V{\'i}ctor and Torelli, Sara},
  title = {General Infinitesimal Variations of Hodge Structure of Ample Curves in Surfaces},
  journal = {arXiv e-prints},
  year = {2024},
  month = {2},
  eprint = {2402.15158},
  archivePrefix = {arXiv},
  primaryClass = {math.AG},
  doi = {10.48550/arXiv.2402.15158},
  url = {https://arxiv.org/abs/2402.15158}
}

@article{Loyola2024,
  author = {Loyola, Roberto Villaflor},
  title = {Toric Differential Forms and Periods of Complete Intersections},
  journal = {Journal of Algebra},
  volume = {643},
  year = {2024},
  pages = {86--118},
  doi = {https://doi.org/10.1016/j.jalgebra.2023.12.021},
  url = {https://www.sciencedirect.com/science/article/pii/S0021869323006476}
}

@article{Cox95,
  author = {Cox, David A.},
  title = {Toric residues},
  journal = {Arkiv f\"{o}r Matematik},
  volume = {34},
  number = {1},
  year = {1995},
  pages = {73--96},
  doi = {10.1007/BF02559508},
  url = {https://projecteuclid.org/journals/arkiv-for-matematik/volume-34/issue-1/Toric-residues/10.1007/BF02559508.full}
}

@book{huybrechts2004complex,
  author = {Huybrechts, Daniel},
  title = {Complex Geometry: An Introduction},
  year = {2004},
  publisher = {Springer},
  address = {Berlin, Heidelberg},
  isbn = {978-3-540-21290-4},
  doi = {10.1007/b137952}
}

@article{CarlsonGreenGriffithsHarris1983,
  author = {Carlson, James and Green, Mark and Griffiths, Phillip and Harris, Joe},
  title = {Infinitesimal Variations of Hodge Structure (I)},
  journal = {Compositio Mathematica},
  volume = {50},
  number = {2-3},
  year = {1983},
  pages = {109--205},
  url = {http://www.numdam.org/item/CM_1983__50_2-3_109_0/},
  mrnumber = {720288},
  zbl = {0531.14006}
}

@article{Griffiths68,
    author = {Griffiths, Phillip} ,
    title = {Periods of integrals on algebraic manifolds, I,II} ,
    journal = {Ameri. J. Math} ,
    volume = {50},
    year = {1968},
    pages= {568--626, 805--865}
}

@article{griffiths1983,
  author = {Griffiths, Phillip A.},
  title = {Infinitesimal variations of hodge structure (III): determinantal varieties and the infinitesimal invariant of normal functions},
  journal = {Compositio Mathematica},
  volume = {50},
  number = {2-3},
  pages = {267--324},
  year = {1983},
  url = {http://www.numdam.org/item/CM_1983__50_2-3_267_0/}
}

@book{arbarello1985,
  author = {Arbarello, Enrico and Cornalba, Maurizio and Griffiths, Phillip A. and Harris, Joseph Daniel},
  title = {Geometry of Algebraic Curves, Volume I},
  series = {Grundlehren der mathematischen Wissenschaften},
  volume = {267},
  publisher = {Springer-Verlag},
  address = {New York},
  year = {1985}
}

@article{colombo2024,
  author = {Colombo, Elisabetta and Frediani, Paola and Pirola, Gian Pietro},
  title = {Asymptotic directions in the moduli space of curves},
  journal = {arXiv preprint arXiv:2405.08641},
  year = {2024},
  url = {https://arxiv.org/abs/2405.08641}
}

@article{raviolo2014,
  author = {Raviolo, Emanuele},
  title = {A note on Griffiths infinitesimal invariant for curves},
  journal = {Annali di Matematica Pura ed Applicata (1923 -)},
  volume = {193},
  number = {2},
  pages = {551--559},
  year = {2014},
  doi = {10.1007/s10231-012-0290-x},
  url = {https://link.springer.com/article/10.1007/s10231-012-0290-x}
}

@article{beorchia2020,
  author = {Beorchia, Valentina and Pirola, Gian Pietro and Zucconi, Francesco},
  title = {Trigonal Deformations of Rank One and Jacobians},
  journal = {International Mathematics Research Notices},
  volume = {2021},
  number = {12},
  pages = {8868--8884},
  year = {2020},
  doi = {10.1093/imrn/rnz216},
  url = {https://academic.oup.com/imrn/article/2021/12/8868/5573589}
}

@article{favale2018,
  author = {Favale, Filippo Francesco and Naranjo, Juan Carlos and Pirola, Gian Pietro},
  title = {On the Xiao conjecture for plane curves},
  journal = {Geometriae Dedicata},
  volume = {195},
  pages = {193--201},
  year = {2018},
  doi = {https://doi.org/10.1007/s10711-017-0283-4},
  url = {https://link.springer.com/article/10.1007/s10711-017-0283-4}
}

@article{codogni2023,
  author = {Codogni, Giulio and Gonz{\'a}lez-Alonso, V{\'i}ctor and Torelli, Sara },
  title = {Rigidity of modular morphisms via Fujita decomposition},
  journal = {arXiv preprint arXiv:2305.04525},
  year = {2023},
  url = {https://arxiv.org/abs/2305.04525}
}

@article{gonzalezalonso2021,
  author = {Gonz{\'a}lez-Alonso, V{\'i}ctor and Torelli, Sara},
  title = {Families of curves with Higgs field of arbitrarily large kernel},
  journal = {Bulletin of the London Mathematical Society},
  volume = {53},
  number = {2},
  pages = {493--506},
  year = {2021},
  doi = {10.1112/blms.12437},
  url = {https://doi.org/10.1112/blms.12437}
}

@article{lee2016,
  author = {Lee, Yongnam and Pirola, Gian Pietro},
  title = {On rational maps from the product of two general curves},
  journal = {Annali della Scuola Normale Superiore di Pisa - Classe di Scienze},
  volume = {16},
  number = {4},
  pages = {1139--1152},
  year = {2016},
  url = {https://arxiv.org/abs/1411.4263}
}

@article{green_period_map_1985,
  author       = {Green, Mark L.},
  title        = {The period map for hypersurface sections of high degree of an arbitrary variety},
  journal      = {Compositio Mathematica},
  volume       = {55},
  number       = {2},
  date         = {1985},
  pages        = {135–156},
  publisher    = {Martinus Nijhoff Publishers},
  url          = {https://www.numdam.org/item/CM_1985__55_2_135_0/},
  language     = {en},
  mrnumber     = {795711},
  zbl          = {0588.14004},
}

@Misc{M2,
          author = {Grayson, Daniel R. and Stillman, Michael E.},
          title = {Macaulay2, a software system for research in algebraic geometry},
          howpublished = {Available at \url{http://www2.macaulay2.com}}
        }

@book {MR4898511,
    AUTHOR = {Eisenbud, David and Harris, Joe},
     TITLE = {The practice of algebraic curves---a second course in
              algebraic geometry},
    SERIES = {Graduate Studies in Mathematics},
    VOLUME = {250},
 PUBLISHER = {American Mathematical Society, Providence, RI},
      YEAR = {[2024] \copyright 2024},
     PAGES = {xv+413},
      ISBN = {978-1-4704-7637-3},
   MRCLASS = {14Hxx (14-01 14-02)},
  MRNUMBER = {4898511},
}

@book {MR2290010,
    AUTHOR = {Cox, David and Little, John and O'Shea, Donal},
     TITLE = {Ideals, varieties, and algorithms},
    SERIES = {Undergraduate Texts in Mathematics},
   EDITION = {Third},
      NOTE = {An introduction to computational algebraic geometry and
              commutative algebra},
 PUBLISHER = {Springer, New York},
      YEAR = {2007},
     PAGES = {xvi+551},
      ISBN = {978-0-387-35650-1; 0-387-35650-9},
   MRCLASS = {13P10 (13-01 14-01 14Qxx 68W30)},
  MRNUMBER = {2290010},
       DOI = {10.1007/978-0-387-35651-8},
       URL = {https://doi.org/10.1007/978-0-387-35651-8},
}

@misc{sernesi2025ivhsnodalplanecurves,
      title={IVHS of nodal plane curves}, 
      author={Edoardo Sernesi},
      year={2025},
      eprint={2512.12316},
      archivePrefix={arXiv},
      primaryClass={math.AG},
      url={https://arxiv.org/abs/2512.12316}, 
}
\vspace*{1cm}
{\raggedright
\small Department of Mathematics, University of California San Diego, La Jolla, CA 92093 \\
\small Email: \href{mailto:your.email@example.com}{jiz185@ucsd.edu}
}

\end{document}